\documentclass{amsart}
\usepackage{amsfonts}
\usepackage{amsmath}
\usepackage{amssymb}
\usepackage{amscd}
\usepackage{amstext}
\usepackage{amsthm}
\usepackage{enumerate}
\usepackage{color}

\theoremstyle{plain}
\newtheorem{lem}{Lemma}[section]
\newtheorem{prop}{Proposition}[section]
\newtheorem{thm}{Theorem}[section]

\newtheorem{mainthm}{Theorem}

\theoremstyle{definition}
\newtheorem{defn}{Definition}[section]

\theoremstyle{remark}
\newtheorem{ex}{Example}[section]
\newtheorem{rem}{Remark}[section]

\begin{document}
\title[Convexity properties of generalized moment maps]
{Convexity properties \\of \\generalized moment maps}
\author{Yasufumi Nitta}
\address{Department of Mathematics, Graduate school of science, Osaka University. 1-1 Machikaneyama, Toyonaka, Osaka 560-0043 Japan.}
\date{}
\maketitle
\thispagestyle{empty}

\begin{abstract}
In this paper, we consider generalized moment maps for 
Hamiltonian actions on $H$-twisted 
generalized complex manifolds introduced by Lin and Tolman \cite{Lin}. 
The main purpose of this paper is to show convexity and connectedness 
properties for generalized moment maps. 
We study Hamiltonian torus actions on 
compact $H$-twisted generalized complex manifolds and prove 
that all components of the generalized moment map are Bott-Morse functions. 
Based on this, we shall show that the generalized moment maps have 
a convex image and connected fibers. Furthermore, by applying the arguments 
of Lerman, Meinrenken, Tolman, and Woodward \cite{Ler2} 
we extend our results to the case of 
Hamiltonian actions of general compact Lie groups 
on $H$-twisted generalized complex orbifolds. 
\end{abstract}

\section{Introduction}
The notion of ($H$-twisted) generalized complex structures was introduced by 
Hitchin \cite{Hi} inspired by physical motivations. 
It provides us with a unifying framework for both complex and symplectic 
geometry and with a useful geometric language for understanding 
some recent development in string theory. 
The associated notion of $H$-twisted generalized K${\rm \ddot a}$hler 
structures was introduced by Gualtieri \cite{Gua}, 
showing that this notion is essentially equivalent to that of 
bihermitian structures. This equivalence was first observed by 
physicists in their study \cite{GHR} of a super-symmetric nonlinear 
sigma model. 

For Hamiltonian group actions on manifolds, moment maps are a very 
useful tool in geometry. In generalized complex geometry, 
Lin and Tolman studied the notions of Hamiltonian actions and 
generalized moment maps for actions of compact Lie groups on 
$H$-twisted generalized complex manifolds \cite{Lin}, 
and established a reduction theorem. 
In the present paper we study the convexity properties of 
generalized moment maps for Hamiltonian actions. Both convexity and connectedness for 
moment maps in symplectic geometry were studied by Atiyah \cite{At} 
and Guillemin-Sternberg \cite{GS} in the case of torus actions on 
compact symplectic manifolds. We here consider Hamiltonian 
torus actions on compact connected $H$-twisted generalized 
complex manifolds and prove such convexity and 
connectedness for generalized moment maps (cf. Sections 2 and 3). 
\begin{mainthm}
Let an $m$-dimensional torus $T^{m}$ act on a compact connected 
$H$-twisted generalized complex manifold $(M, \mathcal{J})$ in a 
Hamiltonian way with a generalized moment map 
$\mu : M \to \mathfrak{t}^{*}$ and a moment one 
form $\alpha \in \Omega^{1}(M; \mathfrak{t}^{*})$. 
Then 
\begin{enumerate}
\item the levels of $\mu$ are connected, 
\item the image of $\mu$ is convex, and 
\item the fixed points of the action form a finite 
union of connected generalized complex submanifolds $C_{1}, \cdots, C_{N}$:
\begin{equation*}
{\rm Fix}(T^{m}) = \bigcup_{i=1}^{N}C_{i}. 
\end{equation*}
On each component the generalized moment map 
$\mu$ attains a constant: $\mu(C_{i}) = \{a_{i}\}$, and the image of 
$\mu$ is the convex hull of the images $a_{1}, \cdots, a_{N}$ 
of the fixed points, that is, 
\begin{equation*}
\mu(M) = 
\left\{ \sum_{i=1}^{N}\lambda_{i}a_{i}\ |\ \sum_{i=1}^{N}\lambda_{i}=1, \lambda_{i} \geq 0 \right\}. 
\end{equation*}
\end{enumerate}
\end{mainthm}

For the proof of the theorem, we need to show that 
all components of the generalized moment map are Bott-Morse functions, 
that is, the function $\mu^{\xi} : M \to \mathbb{R}$ is a Bott-Morse function for all $\xi \in \mathfrak{t}$ (cf. Proposition \ref{BM}). This is crucial in the proof, and is obtained by 
the maximum principle for pseudoholomorphic functions on 
almost complex manifolds. 

In the latter part of this paper, we shall extend our results to the case of 
general compact Lie group actions on $H$-twisted 
generalized complex orbifolds under the assumption of {\it weak nondegeneracy} (cf. Definition \ref{weak_nondegeneracy}) for generalized moment maps, where weak nondegeneracy is always the case for compact orbifolds. 
Recall that the non-abelian convexity theorem 
in symplectic geometry was proved by Kirwan 
\cite{Kir} and Lerman-Meinrenken-Tolman-Woodward \cite{Ler2}. 
A subset $\Delta$ of a vector 
space $V$ is {\it polyhedral} if it is an intersection of finitely many 
closed half-spaces, and is {\it locally polyhedral} 
if for each point $p \in \Delta$ there exist a neighborhood $U$ of $p$ in $V$ 
and a polyhedral set $P$ in $V$ such that $U \cap \Delta = U \cap P$. 
Then we obtain: 
\begin{mainthm}
Let $(M, \mathcal{J})$ be a connected $H$-twisted generalized complex 
orbifold with a Hamiltonian action of a compact connected Lie group $G$, 
a proper generalized moment map $\mu : M \to \mathfrak{g}^{*}$, and a 
moment one form $\alpha \in \Omega^{1}(M; \mathfrak{g}^{*})$. Suppose that 
the generalized moment map $\mu$ has weak nondegeneracy. 
\begin{enumerate}
\item 
Let $\mathfrak{t}_{+}^{*}$ be a closed Weyl chamber for the Lie group 
considered as a subset of $\mathfrak{g}^{*}$. The moment set 
$\Delta = \mu(M) \cap \mathfrak{t}_{+}^{*}$ is a convex locally polyhedral 
set. In particular, if $M$ is compact then $\Delta$ is a convex polytope. 
\item The levels of $\mu$ are connected. 
\end{enumerate}
\end{mainthm}

Let us explain the real meaning of the convexity property for 
generalized moment maps. 
In general, an $H$-twisted generalized complex structure 
is of an intermediate type, i.e., it is neither 
a complex structure nor a symplectic structure. Then the manifold is 
locally fibered over a complex base space such that symplectic 
structures appear in the fiber directions. 
The generalized moment map is thought of as a ``relative version" of the ordinary moment map. Now our theorems on generalized moment maps show not 
only the convexity of the image of each fiber but also 
the convexity of the global image of the generalized moment maps 
(cf. Section 4.4). 

This paper is organized as follows. In section 2 we briefly review 
the theory of generalized complex structures and generalized 
K${\rm \ddot a}$hler structures. Furthermore we introduce 
generalized complex submanifolds of $H$-twisted generalized 
complex manifolds in the sense of Ben-Bassat and Boyarchenko 
\cite{B}. In section 3 we consider the notion of 
generalized moment maps \cite{Lin} for Hamiltonian actions on 
$H$-twisted generalized complex manifolds and prove that 
all components of the generalized moment map are Bott-Morse functions. 
After that we shall give a proof of Theorem A. 
Finally in the last section, we give a proof of Theorem B. 
Our proof follows that of the non-abelian convexity and connectedness 
properties in symplectic geometry by 
Lerman-Meinrenken-Tolman-Woodward in \cite{Ler2}. 
To complete our proof, we need an additional constraint ``weak nondegeneracy" for generalized moment maps and a generalized complex geometry 
analogue of the cross-section theorem in \cite{Ler2} (cf. Theorem \ref{cross2}). 

\section{Generalized complex structures}
We recall the basic theory of generalized complex 
structures; see \cite{Gua} for the details. Throughout this paper, 
we assume that all manifolds and orbifolds are connected. 

\subsection{Generalized complex structures}
Given a closed differential $3$-form $H$ on an $n$-dimensional manifold $M$, 
we define the $H$-twisted Courant bracket of sections of 
the direct sum $TM \oplus T^{*}M$ of the tangent and cotangent bundles by 
\begin{equation*}
[X + \alpha, Y + \beta]_{H} = 
[X, Y] + \mathcal{L}_{X}\beta - \mathcal{L}_{Y}\alpha 
- \frac{1}{2}d\left( \beta(X) - \alpha(Y) \right) + i_{Y}i_{X}H, 
\end{equation*}
where $\mathcal{L}_{X}$ denotes the Lie derivative along a vector 
field $X$. The vector bundle $TM \oplus T^{*}M$ is also endowed 
with a natural inner product of signature $(n, n)$: 
\begin{equation*}
\langle X + \alpha, Y + \beta \rangle = \frac{1}{2}(\beta(X) + \alpha(Y)). 
\end{equation*}

\begin{defn}
Let $M$ be a manifold and $H$ be a closed $3$-form on $M$. 
A generalized almost complex structure on $M$ is a complex 
structure $\mathcal{J}$ on the bundle $TM \oplus T^{*}M$ 
which preserves the natural inner product. 
If sections of the $\sqrt{-1}$-eigenspace $L$ of $\mathcal{J}$ is closed 
under the $H$-twisted Courant bracket, then $\mathcal{J}$ is called an 
$H$-twisted generalized complex structure of $M$. If $H = 0$, we call it 
simply a generalized complex structure. 
\end{defn}

An $H$-twisted generalized complex structure $\mathcal{J}$ 
can be fully described in terms of its $\sqrt{-1}$-eigenbundle $L$, 
which is a maximal isotropic subbundle of 
$(TM \oplus T^{*}M) \otimes \mathbb{C}$ satisfying 
$L \cap \bar L = \{ 0 \}$ and to be closed under 
the $H$-twisted Courant bracket. 
For the natural projection $\pi : (TM \oplus T^{*}M) \otimes \mathbb{C} 
\rightarrow TM \otimes \mathbb{C}$, the codimension of $\pi(L)$ in 
$TM \otimes \mathbb{C}$ is called the type of $\mathcal{J}$ and written 
by ${\rm type}(\mathcal{J})$. 

\begin{ex}[Complex structures (type $n$)]
Let $J$ be a complex structure on an $n$-dimensional complex manifold $M$. 
Consider the endomorphism 
\begin{equation*}
\mathcal{J}_{J}=
\left(
\begin{array}{ccc}
 J  & 0      \\
 0  & -J^{*} \\
\end{array}
\right), 
\end{equation*}
where the matrix is written with respect to the direct sum 
$TM \oplus T^{*}M$. Then $\mathcal{J}_{J}$ is a 
generalized complex structure of type $n$ on $M$; 
the $\sqrt{-1}$-eigenspace of $\mathcal{J}_{J}$ is 
$L_{J} = T_{1,0}M \oplus T^{0,1}M$, where $T_{1,0}M$ is the 
$\sqrt{-1}$-eigenspace of $J$. 
\end{ex}

\begin{ex}[Symplectic structures (type 0)]
Let $M$ be a symplectic manifold with a symplectic structure $\omega$, 
viewed as a skew-symmetric isomorphism 
$\omega :TM \to T^{*}M$ via the interior product 
$X \mapsto i_{X}\omega$. Consider the endomorphism 
\begin{equation*}
\mathcal{J}_{\omega}=
\left(
\begin{array}{ccc}
 0       & -\omega^{-1} \\
 \omega  & 0            \\
\end{array}
\right). 
\end{equation*}
Then $\mathcal{J}_{\omega}$ is a 
generalized complex structure on $M$ of type $0$; 
the $\sqrt{-1}$-eigenspace of $\mathcal{J}_{\omega}$ is given by 
\begin{equation*}
L_{\omega} = \{ X - \sqrt{-1}i_{X}\omega\ |\ X \in TM \otimes \mathbb{C} \}.
\end{equation*}
\end{ex}

\begin{ex}[$B$-field shift]
Let $(M, \mathcal{J})$ be an $H$-twisted generalized complex manifold 
and $B \in \Omega^{2}(M)$ be a closed $2$-form on $M$. Then the endomorphism 
\begin{equation*}
\mathcal{J}_{B}=
\left(
\begin{array}{ccc}
 1 & 0 \\
 B & 1 \\
\end{array}
\right)
\mathcal{J}
\left(
\begin{array}{ccc}
 1   & 0 \\
 -B  & 1 \\
\end{array}
\right)
\end{equation*}
is also an $H$-twisted generalized complex structure. It is called the 
$B$-field shift of $\mathcal{J}$. The type of $\mathcal{J}_{B}$ coincides 
with that of $\mathcal{J}$. Indeed, the $\sqrt{-1}$ eigenspace $L_{B}$ of 
$\mathcal{J}_{B}$ can be written as 
\begin{equation*}
L_{B} = \{ X + \alpha + i_{X}B\ |\ X + \alpha \in L \}, 
\end{equation*}
where $L$ is the $\sqrt{-1}$ eigenspace of $\mathcal{J}$. 
\end{ex}

The type of an $H$-twisted generalized complex structure is not required to 
be constant along the manifold and it may jump along loci. 
Gualtieri constructed a generalized complex structure on 
$\mathbb{CP}^{2}$ which is type 2 along a cubic curve and type 0 
outside the cubic curve. The detailed construction can be seen in 
\cite{Gua}. 

Next we describe the notions of $H$-twisted generalized complex structures 
from the view point of differential forms. For the details, see \cite{Gua}. 
Let $(M, \mathcal{J})$ be a $2n$-dimensional $H$-twisted generalized 
complex manifold with its $\sqrt{-1}$-eigenspace $L$. 
Recall that the exterior algebra $\wedge^{\bullet}T^{*}M$ 
carries a natural spin representation for the metric bundle 
$TM \oplus T^{*}M$; the Clifford action of $X + \alpha \in TM \oplus T^{*}M$ 
on $\varphi \in \wedge^{\bullet}T^{*}M$ is given by 
\begin{equation*}
(X + \alpha) \cdot \varphi = i_{X}\varphi + \alpha \wedge \varphi. 
\end{equation*}
The annihilator $K$ of $L$ by the spin 
representation forms a complex line subbundle of the complex spinors 
$\wedge^{\bullet}T^{*}M \otimes \mathbb{C}$. We call $K$ the 
canonical line bundle of $\mathcal{J}$: 
\begin{equation*}
K = \{ \varphi \in \wedge^{\bullet}T^{*}M \otimes \mathbb{C}\ |\ 
(X + \alpha) \cdot \varphi = 0\ (\forall X + \alpha \in L) \}. 
\end{equation*}
Then the $\sqrt{-1}$-eigenspace $L$ can also be viewed as 
an annihilator of $K$. 

Conversely, for a complex spinor 
$\varphi \in \wedge^{\bullet}T^{*}M \otimes \mathbb{C}$, 
consider $L_{\varphi}$ the annihilator of $\varphi$: 
\begin{equation*}
L_{\varphi} = \{ X + \alpha \in (TM \oplus T^{*}M) \otimes \mathbb{C}\ |\ 
(X + \alpha) \cdot \varphi = 0 \}. 
\end{equation*}
Then the subspace $L_{\varphi} \subset (TM \oplus T^{*}M) \otimes \mathbb{C}$ is 
always isotropic. If $L_{\varphi}$ is maximal isotropic, 
$\varphi$ is called a complex pure spinor. 
A necessary and sufficient condition that $\varphi$ is pure 
can be described as follows. We call a complex differential 
$k$-form $\Omega$ to be decomposable if it has the algebraic form 
$\Omega = \theta^{1} \wedge \cdots \wedge \theta^{k}$ at each point, 
where $\theta^{1}, \cdots, \theta^{k}$ are linearly independent 
complex $1$-forms. 
Then the spinor $\varphi$ is pure if and only if it can be written 
locally as 
\begin{equation*}
\varphi = e^{B + \sqrt{-1}\omega} \wedge \Omega, 
\end{equation*}
where $B$ and $\omega$ are real 2-forms and 
$\Omega$ is a complex decomposable $k$-form. The condition 
$L_{\varphi} \cap \bar L_{\varphi} = \{ 0 \}$ is equivalent to 
an additional constraint on 
$\varphi$: 
\begin{equation*}
\omega^{2(n-k)} \wedge \Omega \wedge \bar \Omega \not = 0. 
\end{equation*}
A complex pure spinor $\varphi$ which satisfies the condition above 
is said to be nondegenerate. 
If a complex differential form 
$\varphi \in \Omega^{\bullet} \otimes \mathbb{C}$ 
is a nondegenerate complex pure spinor at every point on $M$, 
then we have 
$(TM \oplus T^{*}M) \otimes \mathbb{C} = L_{\varphi} \oplus \bar L_{\varphi}$, 
and $L_{\varphi}$ defines a generalized almost complex structure 
on $M$. For each point, the integer $k$ defined above coincides with 
the type of the generalized almost complex structure. 
The canonical line bundle is generated by 
the complex pure spinor $\varphi$. 

Finally as shown in \cite{Gua}, the involutivity of $L_{\varphi}$ 
under the Courant bracket is equivalent to the condition 
that there exist a section $X + \alpha$ of 
$(TM \oplus T^{*}M) \otimes \mathbb{C}$ such that
\begin{equation*}
d\varphi + H \wedge \varphi = (X + \alpha) \cdot \varphi. 
\end{equation*}

\subsection{Generalized K${\rm \ddot a}$hler structures}
We briefly review the notion of 
$H$-twisted generalized K${\rm \ddot a}$hler structures. 

\begin{defn}
Let $M$ be a manifold and $H$ a closed $3$-form on $M$. 
An $H$-twisted generalized K${\rm \ddot a}$hler structure on $M$ 
is a pair of commuting $H$-twisted generalized complex structures 
$(\mathcal{J}_{1}, \mathcal{J}_{2})$ so that 
$\mathcal{G} = -\mathcal{J}_{1} \mathcal{J}_{2}$ is 
a positive definite metric, 
that is, $\mathcal{G}^{2} = {\rm id}$, $\mathcal{G}$ preserves 
the natural inner product and 
$\mathcal{G}(X + \alpha, X + \alpha) := 
\langle \mathcal{G}(X + \alpha), X + \alpha \rangle > 0$ 
for all non-zero $X + \alpha \in TM \oplus T^{*}M$. 
\end{defn}

\begin{ex}
Let $(M, g, J)$ be a K${\rm \ddot a}$hler manifold and $\omega = gJ$ 
be the K${\rm \ddot a}$hler form. As seen in the examples above, 
$J$ and $\omega$ induce generalized complex structures $\mathcal{J}_{J}$ 
and $\mathcal{J}_{\omega}$ respectively. Moreover, we see that 
$\mathcal{J}_{J}$ commutes with $\mathcal{J}_{\omega}$, and 
\begin{equation*}
\mathcal{G} = -\mathcal{J}_{J}\mathcal{J}_{\omega} = 
\left(
\begin{array}{ccc}
 0 & g^{-1} \\
 g & 0 \\
\end{array}
\right)
\end{equation*}
is a positive definite metric on $TM \oplus T^{*}M$. Hence 
$(\mathcal{J}_{J}, \mathcal{J}_{\omega})$ is a generalized 
K${\rm \ddot a}$hler structure on $M$. 
\end{ex}
\begin{ex}
Let $(\mathcal{J}_{1}, \mathcal{J}_{2})$ be an $H$-twisted generalized 
K${\rm \ddot a}$hler structure, and $B$ be a closed $2$-form on $M$. 
Then $((\mathcal{J}_{1})_{B}, (\mathcal{J}_{2})_{B})$ is also an 
$H$-twisted generalized K${\rm \ddot a}$hler structure. It is called 
the $B$-field shift of $(\mathcal{J}_{1}, \mathcal{J}_{2})$. 
\end{ex}

In \cite{Gua}, a characterization of $H$-twisted generalized 
K${\rm \ddot a}$hler structures was given in terms of Hermitian 
geometry, which is represented below. 

\begin{thm}[M. Gualtieri, \cite{Gua}]
For each $H$-twisted generalized K${\rm \ddot a}$hler structure 
$(\mathcal{J}_{1}, \mathcal{J}_{2})$, there exists a uniquely determined 
$2$-form $b$, a Riemannian metric $g$ and two orthogonal 
complex structures $J_{\pm}$ such that 
\begin{equation*}
\mathcal{J}_{1,2} = \frac{1}{2}
\left(
\begin{array}{ccc}
 1 & 0 \\
 b & 1 \\
\end{array}
\right)\left(
\begin{array}{ccc}
 J_{+} \pm J_{-}           & -(\omega^{-1}_{+} \mp \omega^{-1}_{-}) \\
 \omega_{+} \mp \omega_{-} & -(J^{*}_{+} \pm J^{*}_{-})            \\
\end{array}
\right)\left(
\begin{array}{ccc}
 1  & 0 \\
 -b & 1 \\
\end{array}
\right), 
\end{equation*}
where $\omega_{\pm} = gJ_{\pm}$ with the condition 
\begin{equation}\label{torsion}
d^{c}_{-}\omega_{-} = -d^{c}_{+}\omega_{+} = H + db. 
\end{equation}
Conversely, any quadruple $(g, b, J_{\pm})$ satisfying the condition 
$(\ref{torsion})$ defines an $H$-twisted generalized 
K${\rm \ddot a}$hler structure. 
\end{thm}

Not every $H$-twisted generalized complex manifold admits an 
$H$-twisted generalized K${\rm \ddot a}$hler structure. However, the 
following lemma claims that every $H$-twisted generalized complex 
manifold always admits a ``generalized almost K${\rm \ddot a}$hler structure". 
This is a generalized complex geometry analogue of the fact 
that a symplectic manifold admits an almost complex structure 
which is compatible with the symplectic structure. 

\begin{lem}\label{compatible1}
Let $(M, \mathcal{J})$ be an $H$-twisted generalized complex 
manifold. Then there exists a generalized almost complex structure 
$\mathcal{J}^{\prime}$ which is compatible with $\mathcal{J}$, that is, 
$\mathcal{J}^{\prime}$ is a generalized almost complex structure which 
commutes with $\mathcal{J}$, and 
$\mathcal{G} = -\mathcal{J}\mathcal{J}^{\prime}$ is a positive 
definite metric. 
\end{lem}

\begin{proof}
Choose a Riemannian metric $g$ on $M$ and put 
\begin{equation*}
\tilde{\mathcal{G}} = 
\left(
\begin{array}{ccc}
 0 & g^{-1} \\
 g & 0 \\
\end{array}
\right). 
\end{equation*}
Then $\tilde{\mathcal{G}}$ is a positive definite metric on $TM \oplus T^{*}M$. 
Define a symplectic structure $\mathcal{W}$ on $TM \oplus T^{*}M$ 
by 
\begin{equation*}
\mathcal{W}(X + \alpha, Y + \beta) = 
\langle \mathcal{J}(X + \alpha), Y + \beta \rangle . 
\end{equation*}
Since $\tilde{\mathcal{G}}$ and $\mathcal{W}$ are nondegenerate, 
there exists an endomorphism $\mathcal{A}$ on $TM \oplus T^{*}M$ 
which satisfies 
\begin{equation*}
\mathcal{W}(X + \alpha, Y + \beta) = 
\tilde{\mathcal{G}}(\mathcal{A}(X + \alpha), Y + \beta) 
\end{equation*}
for all $X + \alpha, Y + \beta \in TM \oplus T^{*}M$. 
The endomorphism $\mathcal{A}$ is skew-symmetric 
with respect to the positive definite metric 
$\tilde{\mathcal{G}}$ because 
$\mathcal{W} = \tilde{\mathcal{G}}\mathcal{A}$ is an alternating $2$-form 
on $TM \oplus T^{*}M$. 
Let $\mathcal{A}^{*}$ be the adjoint operator of 
$\mathcal{A}$ with respect to $\tilde{\mathcal{G}}$. 
Since $\mathcal{A}$ is invertible, 
$\mathcal{A}\mathcal{A}^{*} = -\mathcal{A}^{2}$ 
is symmetric and positive, that is, 
$(\mathcal{A}\mathcal{A}^{*})^{*} = \mathcal{A}\mathcal{A}^{*}$ and 
\begin{equation*}
\tilde{\mathcal{G}}(\mathcal{A}\mathcal{A}^{*}(X + \alpha), X + \alpha ) > 0
\end{equation*}
for all non-zero $X + \alpha \in TM \oplus T^{*}M$. 
Hence we can define $\sqrt{\mathcal{A}\mathcal{A}^{*}}$ the square root 
of $\mathcal{A}\mathcal{A}^{*}$. Here 
$\sqrt{\mathcal{A}\mathcal{A}^{*}}$ is also symmetric and positive definite. 

Let $\mathcal{J}^{\prime}$ be an endomorphism on $TM \oplus T^{*}M$ defined by 
\begin{equation*}
\mathcal{J}^{\prime} = (\sqrt{\mathcal{A}\mathcal{A}^{*}})^{-1}\mathcal{A}. 
\end{equation*}
Since $\mathcal{A}$ commutes with 
$\sqrt{\mathcal{A}\mathcal{A}^{*}}$, $\mathcal{J}^{\prime}$ 
commutes with both $\mathcal{A}$ and $\sqrt{\mathcal{A}\mathcal{A}^{*}}$. 
Hence we obtain $(\mathcal{J}^{\prime})^{2} = -{\rm id}$. 
By the definition of $\mathcal{A}$, we have 
$\mathcal{A}\mathcal{J} = -\mathcal{J}\mathcal{A}^{-1}$ and hence 
$\mathcal{J}$ commutes with $\sqrt{\mathcal{A}\mathcal{A}^{*}}$. 
In particular, we see that $\mathcal{J}^{\prime}$ commutes 
with $\mathcal{J}$. Moreover, since $\mathcal{J}^{\prime}$ is 
orthogonal with respect to $\tilde{\mathcal{G}}$, we can check easily 
that $\mathcal{J}^{\prime}$ is orthogonal with respect to the natural 
inner product on $TM \oplus T^{*}M$. Hence $\mathcal{J}^{\prime}$ is a generalized almost complex structure on $M$ which commutes with 
$\mathcal{J}$. Finally $\mathcal{G} := -\mathcal{J}\mathcal{J}^{\prime}$ is a 
positive definite metric on $TM \oplus T^{*}M$ since 
$\mathcal{G} = \tilde{\mathcal{G}}\sqrt{\mathcal{A}\mathcal{A}^{*}}$. 
This completes the proof. 
\end{proof}

If $\mathcal{J}^{\prime}$ is a generalized almost complex 
structure which is compatible with an $H$-twisted generalized 
complex structure $\mathcal{J}$, then we can apply the argument 
of Gualtieri in \cite{Gua} and construct a $2$-form $b$, 
a Riemannian metric $g$ and two orthogonal almost complex structures 
$J_{\pm}$ which satisfy the equation
\begin{equation}\label{decomp}
\mathcal{J} = \frac{1}{2}
\left(
\begin{array}{ccc}
 1 & 0 \\
 b & 1 \\
\end{array}
\right)\left(
\begin{array}{ccc}
 J_{+} + J_{-}           & -(\omega^{-1}_{+} - \omega^{-1}_{-}) \\
 \omega_{+} - \omega_{-} & -(J^{*}_{+} + J^{*}_{-})            \\
\end{array}
\right)\left(
\begin{array}{ccc}
 1  & 0 \\
 -b & 1 \\
\end{array}
\right). 
\end{equation}
In general, $J_{+}$ and $J_{-}$ may not be integrable. 

\subsection{Generalized complex submanifolds}
Here we introduce the notion of generalized complex submanifolds in the 
sense of Ben-Bassat and Boyarchenko in \cite{B}. 
Let $i: S \to M$ be a submanifold of an $H$-twisted generalized 
complex manifold $(M, \mathcal{J})$. For each $p \in S$, define a subspace 
$(L_{S})_{p} \subset (T_{p}S \oplus T_{p}^{*}S) \otimes \mathbb{C}$ by 
\begin{equation*}
(L_{S})_{p} = 
\{ X + i^{*}\alpha \in (T_{p}S \oplus T_{p}^{*}S) \otimes \mathbb{C}\ |\ 
X + \alpha \in L_{p} \}, 
\end{equation*}
where $L$ is the $\sqrt{-1}$-eigenspace of $\mathcal{J}$. Clearly 
$(L_{S})_{p}$ is an isotropic subspace of 
$(T_{p}S \oplus T_{p}^{*}S) \otimes \mathbb{C}$. 
Furthermore it is easy to see that 
$\dim_{\mathbb{C}} (L_{S})_{p} = \dim S$ 
and hence $(L_{S})_{p}$ is a maximal isotropic subspace. 
However, the distribution $L_{S} := \cup_{p \in S}(L_{S})_{p}$ may not be 
a subbundle of $(TS \oplus T^{*}S) \otimes \mathbb{C}$ in general. 
We refer the reader to \cite{Cou} for a detailed discussion in 
the case of submanifolds of Dirac manifolds. 
In particular, Courant's arguments can be easily adapted to 
give a necessary condition under which $L_{S}$ is a subbundle of 
$(TS \oplus T^{*}S) \otimes \mathbb{C}$ (cf. \cite{Cou}, Theorem 3.1.1), 
and to prove that if $L_{S}$ is a subbundle and $L$ is integrable, 
then so is $L_{S}$ (cf. \cite{Cou}, Corollary 3.1.4). 

\begin{defn}[Ben-Bassat, Boyarchenko, \cite{B}]
We say that $S$ is a generalized complex submanifold
of $M$ if $L_{S}$ is a subbundle of $(TS \oplus T^{*}S) \otimes \mathbb{C}$ 
and satisfies that $L_{S} \cap \bar L_{S} = \{0\}$. 
\end{defn}

If $i : S \to M$ is a generalized 
complex submanifold of an $H$-twisted generalized complex manifold 
$(M, \mathcal{J})$, then $L_{S}$ gives an $i^{*}H$-twisted generalized 
complex structure on $S$. 

\begin{ex}
Let $S$ be a complex submanifold of a complex manifold $(M, J)$. 
Note that $S$ has a natural complex structure induced by $J$. Then 
we have
\begin{equation*}
L_{S} = T_{1, 0}S \oplus T^{0, 1}S, 
\end{equation*}
which is of course a subbundle of $(TS \oplus T^{*}S) \otimes \mathbb{C}$ 
and satisfies $L_{S} \cap \bar L_{S} = \{0\}$. 
Hence $S$ is a generalized complex 
submanifold of $(M, \mathcal{J}_{J})$. The induced generalized complex 
structure of $S$ is the natural generalized complex structure which is 
induced by the complex structure of $S$. 
\end{ex}
\begin{ex}
Let $i : S \to M$ be a symplectic submanifold of 
a symplectic manifold $(M, \omega)$. Then for the 
generalized complex structure $\mathcal{J}_{\omega}$ induced by 
the symplectic structure $\omega$, we have 
\begin{equation*}
L_{S} = \{ X - \sqrt{-1}i_{X}(i^{*}\omega)\ |\ X \in TS \otimes \mathbb{C} \}, 
\end{equation*}
which coincides with the $\sqrt{-1}$-eigenspace of the generalized 
complex structure $\mathcal{J}_{i^{*}\omega}$ induced by the symplectic 
structure $i^{*}\omega$. In particular $S$ is a generalized complex 
submanifold of $(M, \mathcal{J}_{\omega})$. 
\end{ex}
\begin{ex}
Let $S$ be a Lagrangian submanifold of a symplectic 
manifold $(M, \omega)$. Then we can see easily that 
$L_{S} =TS \otimes \mathbb{C}$ and hence $L_{S}$ is a maximal isotropic 
subbundle of $(TS \oplus T^{*}S) \otimes \mathbb{C}$. However, since 
it is clear that 
$L_{S} \cap \bar L_{S} = TS \otimes \mathbb{C} \not= \{ 0 \}$, 
the submanifold $S$ is not a generalized complex submanifold. 
\end{ex}

In general, it may not be easy to determine if 
a given submanifold is a generalized complex submanifold. 
Here we give a simple sufficient condition. 

\begin{prop}\label{prop21}
Let $(M, \mathcal{J})$ be an $H$-twisted generalized complex manifold, 
and $\mathcal{J}^{\prime}$ be a generalized almost complex structure of 
$M$ which is compatible with $\mathcal{J}$. 
We denote $(g, b, J_{\pm})$ the corresponding quadruple. If $i : S \to M$ 
is an almost complex submanifold of $M$ with respect to both 
$J_{+}$ and $J_{-}$, then $S$ is a generalized complex 
submanifold of $(M, \mathcal{J})$. 
\end{prop}
\begin{proof}
Let $T_{1, 0}^{\pm}M \subset TM \otimes \mathbb{C}$ be the 
$\sqrt{-1}$-eigenspace of $J_{\pm}$. Then we can check easily that 
$L$ is given by 
\begin{equation*}
L = 
\{ X + (b + g)(X)\ |\ X \in T_{1, 0}^{+}M \} 
\oplus \{ Y + (b - g)(Y)\ |\ Y \in T_{1, 0}^{-}M \}. 
\end{equation*}
Hence for a given almost complex submanifold $S$, 
the subspace $L_{S}$ can be written as 
\begin{equation*}
L_{S} = 
\{ X + i^{*}(b + g)(X)\ |\ X \in T_{1, 0}^{+}S \} 
\oplus \{ Y + i^{*}(b - g)(Y)\ |\ Y \in T_{1, 0}^{-}S \}, 
\end{equation*}
where $T_{1, 0}^{\pm}S$ is the $\sqrt{-1}$-eigenspace of the restriction of 
$J_{\pm}$ to $S$. In particular, $L_{S}$ is a subbundle of 
$(TS \oplus T^{*}S) \otimes \mathbb{C}$. 
In addition if 
$(X + i^{*}(b + g)(X)) + (Y + i^{*}(b - g)(Y)) \in L_{S} \cap \bar L_{S}$ 
for $X \in T_{1, 0}^{+}S$ and $Y \in T_{1, 0}^{-}S$, then we see that 
$(X + (b + g)(X)) + (Y + (b - g)(Y)) \in L \cap \bar L = \{ 0\} $. 
Thus we obtain $X = Y = 0$, and hence $L_{S} \cap \bar L_{S} = \{ 0 \}$. 
This proves the proposition. 
\end{proof}

\begin{rem}
If $M$ is an orbifold and $H$ is a closed $3$-form on $M$, we can 
define the notions of $H$-twisted generalized complex structures of 
$M$ in usual way. The detailed description is as follows. 
A definition of orbifolds can be seen in \cite{IS} for example. 
Let $M$ be an orbifold and $(V_{i},G_{i}, \pi_{i})_{i \in I}$ 
be a local uniformizing system of $M$. A generalized almost complex 
structure $\mathcal{J}$ of $M$ is a family of endomorphisms 
$\{ \mathcal{J}_{i} : TV_{i} \oplus T^{*}V_{i} \to TV_{i} \oplus T^{*}V_{i} \}_{i \in I}$ 
such that $\mathcal{J}_{i}$ is a generalized almost complex structure 
on $V_{i}$ for each $i \in I$ and they are equivariant under 
the local group actions and compatible with respect to the injections. 
If each $\mathcal{J}_{i}$ is integrable 
with respect to $H$-twisted Courant brackets, then $\mathcal{J}$ is called 
to be integrable and we call it a $H$-twisted generalized complex structure of 
an orbifold $M$. 

In the case that $(M, \mathcal{J})$ is an $H$-twisted generalized complex 
orbifold, we can describe the same notions in section $2$, 
and the assertions in section $2$ still hold in the language of orbifolds. 
\end{rem}

\section{Hamiltonian actions and generalized moment maps}
\subsection{Hamiltonian actions on $H$-twisted generalized complex manifolds}
In this section we recall the definition of 
Hamiltonian actions on $H$-twisted generalized complex manifolds 
given by Lin and Tolman in \cite{Lin}. 

\begin{defn}[Y. Lin and S. Tolman, \cite{Lin}]
Let a compact Lie group $G$ with its Lie algebra $\mathfrak{g}$ 
act on an $H$-twisted generalized complex manifold $(M, \mathcal{J})$ 
preserving $\mathcal{J}$, where $H \in \Omega^{3}(M)$ is a 
$G$-invariant closed $3$-form. The action of $G$ is said to be 
Hamiltonian if there exists a $G$-equivariant 
smooth function $\mu : M \to \mathfrak{g}^{*}$, 
called the generalized moment map, and a $\mathfrak{g}^{*}$-valued 
one form $\alpha \in \Omega^{1}(M; \mathfrak{g}^{*})$, 
called the moment one form, such that 
\begin{enumerate}
\item $\xi_{M} - \sqrt{-1}(d\mu^{\xi} + \sqrt{-1}\alpha^{\xi})$ lies in 
$L$ for all $\xi \in \mathfrak{g}$, where $\xi_{M}$ denotes 
the induced vector field on $M$, and 
\item $i_{\xi_{M}}H = d\alpha^{\xi}$ for all $\xi \in \mathfrak{g}$. 
\end{enumerate}
\end{defn}
Since $L$ is an isotropic subbundle, we have 
\begin{equation*}
\left \langle \xi_{M} - \sqrt{-1}(d\mu^{\xi} + \sqrt{-1}\alpha^{\xi}),\ \xi_{M} - \sqrt{-1}(d\mu^{\xi} + \sqrt{-1}\alpha^{\xi}) \right \rangle = 0 
\end{equation*}
and hence $\iota_{\xi_{M}}\alpha^{\xi} = \iota_{\xi_{M}}d\mu^{\xi} = 0$ 
for each $\xi \in \mathfrak{g}$. 
\begin{ex}
Let a compact Lie group $G$ act on a symplectic manifold $(M, \omega)$ 
preserving the symplectic structure $\omega$, and 
$\mu : M \to \mathfrak{g}^{*}$ 
be an usual moment map, that is, $\mu$ is $G$-equivariant and 
$i_{\xi_{M}}\omega = d\mu^{\xi}$ for all $\xi \in \mathfrak{g}$. 
Then $G$ also preserves $\mathcal{J}_{\omega}$, $\mu$ is also a 
generalized moment map, and $\alpha = 0$ is a moment one form for 
this action. Hence the $G$-action on $(M, \mathcal{J}_{\omega})$ 
is Hamiltonian. 
\end{ex}
\begin{ex}
Let $(M, J)$ be a complex manifold and $G$ act on 
$(M, \mathcal{J}_{J})$ in a Hamiltonian way. Then $G$ also 
preserves the original complex structure $J$. 
Since $L_ {J} = T_{1, 0}M \oplus T^{0,1}M$ and $\xi_{M} \in \pi(L_{J})$, 
we have $\xi_{M} = 0$ for all $\xi \in \mathfrak{g}$. Thus if $G$ is 
connected, the $G$-action on $M$ must be trivial. 
\end{ex}
\begin{ex}
Let $G$ act on an $H$-twisted generalized complex manifold 
$(M, \mathcal{J})$ with a generalized moment map 
$\mu : M \to \mathfrak{g}^{*}$ and a moment 
one form $\alpha \in \Omega^{1}(M; \mathfrak{g}^{*})$. 
If $B \in \Omega^{2}(M)^{G}$ is closed, then $G$ 
acts on $M$ preserving the $B$-field shift of $\mathcal{J}$ with 
generalized moment map $\mu$ and moment one form $\alpha^{\prime}$, where 
$(\alpha^{\prime})^{\xi} = \alpha^{\xi} + i_{\xi_{M}}B$ for all 
$\xi \in \mathfrak{g}$. 
\end{ex}

By the definition, we can treat the notion of 
generalized moment maps as a generalization of the notion of moment maps 
in symplectic geometry. Generalized moment maps are studied by Lin and 
Tolman in \cite{Lin}. In their paper, they showed that a reduction 
theorem for Hamiltonian actions of compact Lie groups on 
$H$-twisted generalized complex manifold holds. We shall use this fact later. 
Note that since $i_{\xi_{M}}H = d\alpha^{\xi}$ and 
$\iota_{\xi_{M}}\alpha^{\xi} = 0$ for each $\xi \in \mathfrak{g}$, 
we can see $H + \alpha$ as an equivariantly closed form. 

\begin{lem}[Y. Lin, S. Tolman, \cite{Lin}]\label{lem31}
Let a compact Lie group $G$ act freely on a manifold $M$. Let $H$ be an 
invariant closed $3$-form and $\alpha$ be an equivariant mapping 
from $\mathfrak{g}$ to $\Omega^{1}(M)$. 
Fix a connection $\theta \in \Omega(M; \mathfrak{g})$. 
Then if $H + \alpha$ is equivariantly closed, 
there exists a natural form $\Gamma \in \Omega^{2}(M)^{G}$ such that 
$i_{\xi_{M}}\Gamma = \alpha^{\xi}$ for all $\xi \in \mathfrak{g}$. 
In particular, 
$H + \alpha + d_{G}\Gamma \in \Omega^{3}(M) \subset \Omega_{G}(M)$, 
where $\Omega_{G}(M)$ is the set of equivariant differential forms of $M$ 
and $d_{G}$ denotes the equivariant exterior differential, 
is closed and basic and so descends to a closed form 
$\tilde{H} \in \Omega^{3}(M/G)$ such that the cohomology class 
$[\tilde{H}]$ is the image of $[H + \alpha]$ under the Kirwan map.
\end{lem}

\begin{thm}[Y. Lin, S. Tolman, \cite{Lin}]\label{reduction}
Let a compact Lie group $G$ act on an $H$-twisted generalized complex
manifold $(M, \mathcal{J})$ in a Hamiltonian way with 
a generalized moment map $\mu : M \to \mathfrak{g}^{*}$ and 
a moment one form $\alpha \in \Omega^{1}(M; \mathfrak{g}^{*})$. 
Let $\mathcal{O}_{a}$ be a coadjoint orbit through $a \in \mathfrak{g}^{*}$ 
so that $G$ acts freely on $\mu^{-1}(\mathcal{O}_{a})$. 
Given a connection on $\mu^{-1}(\mathcal{O}_{a})$, 
the twisted generalized complex quotient $M_{a} = \mu^{-1}(\mathcal{O}_{a})/G$ 
inherits an $\tilde{H}$-twisted generalized complex structure 
$\tilde{\mathcal{J}}$, where $\tilde{H}$ is defined as in Lemma 
\ref{lem31}. Up to $B$-field shift, $\tilde{\mathcal{J}}$ is independent of 
the choice of connection. Finally, for each $p \in \mathcal{O}_{a}$, 
\begin{equation*}
{\rm type}(\tilde{\mathcal{J}})_{[p]} = {\rm type}(\mathcal{J})_{p}. 
\end{equation*}
\end{thm}

In the case that $(M, \mathcal{J})$ is an $H$-twisted generalized complex 
orbifold, we can define the notion of Hamiltonian actions of 
a compact Lie group on $(M, \mathcal{J})$ in usual way. 
In this case, the reduction theorem still holds in the language of 
orbifolds. The detailed statement is as follows. 
Let a compact Lie group $G$ act on an $H$-twisted generalized complex
orbifold $(M, \mathcal{J})$ in a Hamiltonian way with 
a generalized moment map $\mu : M \to \mathfrak{g}^{*}$ and 
a moment one form $\alpha \in \Omega^{1}(M; \mathfrak{g}^{*})$. 
For a coadjoint orbit $\mathcal{O}_{a}$ through $a \in \mathfrak{g}^{*}$, 
suppose that the $G$-action on $\mu^{-1}(\mathcal{O}_{a})$ is locally free. 
Given a connection on $\mu^{-1}(\mathcal{O}_{a})$, 
the twisted generalized complex quotient $M_{a} = \mu^{-1}(\mathcal{O}_{a})/G$ 
is an orbifold and inherits an $\tilde{H}$-twisted generalized complex 
structure $\tilde{\mathcal{J}}$, where $\tilde{H}$ is defined as in 
Lemma \ref{lem31}. Up to $B$-field shift, $\tilde{\mathcal{J}}$ is 
independent of the choice of connection and the type is preserved. 

Before we begin a proof of Theorem A, we shall prove a remarkable fact 
of generalized moment maps. 
At first we prove the following lemmata. 

\begin{lem}\label{compatible2}
Let a compact Lie group $G$ act on an $H$-twisted generalized complex 
manifold $(M, \mathcal{J})$ preserving $\mathcal{J}$. Then there exists a 
$G$-invariant generalized almost complex structure which is compatible 
with $\mathcal{J}$. 
\end{lem}

\begin{proof}
Choose a $G$-invariant Riemannian metric $g$ on $M$ and put 
\begin{equation*}
\mathcal{G} = 
\left(
\begin{array}{ccc}
 0 & g^{-1} \\
 g & 0 \\
\end{array}
\right). 
\end{equation*}
Then $\mathcal{G}$ is a $G$-invariant positive definite metric on 
$TM \oplus T^{*}M$. 
Let $\mathcal{A}$ be an endomorphism on $TM \oplus T^{*}M$ defined by 
$\mathcal{A} = \mathcal{G}^{-1}\mathcal{J}$. 
Since $\mathcal{G}$ and $\mathcal{J}$ are $G$-invariant, 
$\mathcal{A}$ is also $G$-invariant. Now if we define 
\begin{equation*}
\mathcal{J}^{\prime} = (\sqrt{\mathcal{A}\mathcal{A}^{*}})^{-1}\mathcal{A}, 
\end{equation*}
then $\mathcal{J}^{\prime}$ is a generalized almost complex structure 
on $M$ which is compatible with $\mathcal{J}$. Furthermore since 
$\mathcal{A}$ is $G$-invariant, $\mathcal{J}^{\prime}$ is also $G$-invariant. 
This completes the proof. 
\end{proof}

\begin{lem}\label{Fix}
Let an $m$-dimensional torus $T^{m}$ act on an $H$-twisted generalized complex 
manifold $(M, \mathcal{J})$ in a Hamiltonian way with a generalized 
moment map $\mu$ and a moment one form $\alpha$. Then for an arbitrary 
subtorus $G \subset T^{m}$ the fixed point set of $G$-action 
\[
{\rm Fix}(G) = \{ p \in M\ |\ \theta \cdot p = p\ (\forall \theta \in G) \}
\]
is a generalized complex submanifold of $(M, \mathcal{J})$. 
\end{lem}

\begin{proof}
Choose a $G$-invariant generalized almost complex structure 
$\mathcal{J}^{\prime}$ which is compatible with $\mathcal{J}$. 
Then there exists a Riemannian metric $g$, a $2$-form $b$, and 
two orthogonal almost complex structures $J_{\pm}$ which satisfies 
the equation (\ref{decomp}). 
Since $\mathcal{J}$ and $\mathcal{J}^{\prime}$ are $G$-invariant, 
$g$ and $J_{\pm}$ are also $G$-invariant. 
For each $p \in {\rm Fix}(G)$ and $\theta \in G$, 
the differential of the action of $\theta$ at $p$ 
\begin{equation*}
(\theta_{*})_{p} : T_{p}M \to T_{p}M 
\end{equation*}
preserves the almost complex structures $J_{\pm}$. In addition, since 
$G$-action preserves the metric $g$, the exponential mapping 
$\exp_{p} : T_{p}M \to M$ with respect to the metric $g$ 
is equivariant, that is, 
\begin{equation*}
\exp_{p}((\theta_{*})_{p}X) = \theta \cdot \exp_{p}X 
\end{equation*}
for any $\theta \in G$ and $X \in T_{p}M$. This concludes that the 
fixed point of the action of $\theta$ near $p$ corresponds to 
the fixed point of $(\theta_{*})_{p}$ on $T_{p}M$ by the exponential 
mapping, that is, 
\begin{equation*}
T_{p}{\rm Fix}(G) = \bigcap_{\theta \in G}\ker(1 - (\theta_{*})_{p}). 
\end{equation*}
Since $J_{\pm}$ commutes with the endomorphism $(\theta_{*})_{p}$, 
the eigenspace with eigenvalue $1$ of $(\theta_{*})_{p}$ is 
invariant under $J_{\pm}$, and hence an almost complex 
subspace. In particular we see that 
${\rm Fix}(G)$ is a generalized complex submanifold 
of $(M, \mathcal{J})$ by applying Proposition \ref{prop21}.
\end{proof}
Now we consider a Hamiltonian $T^{m}$-action on a compact $H$-twisted generalized complex manifold $(M, \mathcal{J})$ with a generalized moment map $\mu: M \to \mathfrak{g}$ and a moment one form $\alpha \in \Omega^{1}(M; \mathfrak{g}^{*})$, 
and examine the functions $\mu^{\xi} : M \to \mathbb{R}$ for all $\xi \in \mathfrak{g}$. 
The following proposition shows that these are Bott-Morse functions 
with even indices and coindices. 
This is crucial to prove the connectedness of fibers of the generalized moment map. 
In our proof, the maximum principle for pseudoholomorphic functions on 
almost complex manifolds plays a central role. 
The maximum principle for pseudoholomorphic functions on almost complex manifolds is provided by the work of Boothby-Kobayashi-Wang in \cite{BKW}. 
\begin{prop}\label{BM}
Let an $m$-dimensional torus $T^{m}$ act on a compact 
$H$-twisted generalized complex manifold $(M, \mathcal{J})$ 
in a Hamiltonian way with a generalized 
moment map $\mu : M \to \mathfrak{t}^{*}$ and a moment one form 
$\alpha \in \Omega^{1}(M; \mathfrak{t}^{*})$. Then 
$\mu^{\xi}$ is a Bott-Morse function with even index and coindex 
for all $\xi \in \mathfrak{t}$. 
\end{prop}

\begin{proof}
For each $\xi \in \mathfrak{t}$, we denote $T^{\xi}$ the subtorus of 
$T^{m}$ generated by $\xi$. First we shall prove 
that the critical set 
\begin{equation*}
{\rm Crit}(\mu^{\xi}) = \{ p \in M\ |\ (d\mu^{\xi})_{p} = 0 \} 
\end{equation*}
coincides with the fixed point set of $T^{\xi}$-action ${\rm Fix}(T^{\xi})$. 
Choose a $T^{m}$-invariant generalized almost 
complex structure $\mathcal{J}^{\prime}$ which is 
compatible with $\mathcal{J}$. Then $\mathcal{J}$ can be written 
in the form of the equation (\ref{decomp}) for the corresponding 
quadruple $(g, b, J_{\pm})$. Note that the metric $g$ and 
orthogonal almost complex structures $J_{\pm}$ are all $T^{m}$-invariant. 

Since $\xi_{M} - \sqrt{-1}(d\mu^{\xi} + \sqrt{-1}\alpha^{\xi}) \in L$ 
by the definition of Hamiltonian actions,  $(d\mu^{\xi})_{p} = 0$ 
implies $p \in {\rm Fix}(T^{\xi})$. 
In particular we obtain ${\rm Crit}(\mu^{\xi}) \subset {\rm Fix}(T^{\xi})$. 
On the other hand, since 
${\rm Fix}(T^{\xi}) = \{ p \in M\ |\ (\xi_{M})_{p} = 0\}$, 
we see that $\alpha^{\xi} + \sqrt{-1}d\mu^{\xi} \in \bar L$ 
on ${\rm Fix}(T^{\xi})$. Hence there exists a complex vector field 
$X$ on $M$ 
which satisfies that $\alpha^{\xi} + \sqrt{-1}d\mu^{\xi} = g(X)$ 
and $X \in T_{0, 1}^{+}M \cap T_{0, 1}^{-}M$
on ${\rm Fix}(T^{\xi})$ because the $\sqrt{-1}$-eigenspace $L$ 
can be written by
\begin{equation*}
L = 
\{ X + (b + g)(X)\ |\ X \in T_{1, 0}^{+}M \} 
\oplus \{ Y + (b - g)(Y)\ |\ Y \in T_{1, 0}^{-}M \}. 
\end{equation*}
Since the almost complex structures $J_{\pm}$ are orthogonal 
with respect to the metric $g$, we see that 
$\alpha^{\xi} + \sqrt{-1}d\mu^{\xi}$ is a holomorphic $1$-form 
on ${\rm Fix}(T^{\xi})$. 
Moreover, since $\alpha^{\xi}$ is a closed $1$-form on ${\rm Fix}(T^{\xi})$, 
we can view the function $\mu^{\xi}$ locally as an imaginary part of a 
pseudoholomorphic function on an almost complex manifold 
$({\rm Fix}(T^{\xi}), J_{\pm})$. 
By applying the maximum 
principle and compactness of ${\rm Fix}(T^{m})$, 
we see that $\mu^{\xi}$ is constant on each connected component of 
${\rm Fix}(T^{\xi})$. Moreover the gradient of $\mu^{\xi}$ with 
respect to the metric $g$ is 
tangent to ${\rm Fix}(T^{\xi})$ 
because $g$ and $\mu^{\xi}$ are $T^{\xi}$-invariant. 
This shows that ${\rm Fix}(T^{\xi}) \subset {\rm Crit}(\mu^{\xi})$, 
and hence we obtain ${\rm Crit}(\mu^{\xi}) = {\rm Fix}(T^{\xi})$. 
In particular, ${\rm Crit}(\mu^{\xi})$ is 
a generalized complex submanifold of $M$. 

To prove that the function $\mu^{\xi}$ is a Bott-Morse function, we shall calculate the Hessian $\nabla^{2}\mu^{\xi}$ of $\mu^{\xi}$ on ${\rm Crit}(\mu^{\xi})$. 
Since $\xi_{M} - \sqrt{-1}(d\mu^{\xi} + \sqrt{-1}\alpha^{\xi}) \in L$ for each 
$\xi \in \mathfrak{g}$, we have 
\begin{equation}\label{eq1}
\mathcal{J}(\xi_{M} - \sqrt{-1}(d\mu^{\xi} + \sqrt{-1}\alpha^{\xi})) 
= \sqrt{-1}(\xi_{M} - \sqrt{-1}(d\mu^{\xi} + \sqrt{-1}\alpha^{\xi})). 
\end{equation}
In addition, by using the equation (\ref{decomp}) for the $H$-twisted generalized complex structure $\mathcal{J}$, 
for the natural projection 
$\pi: (TM \oplus T^{*}M) \otimes \mathbb{C} \to TM \otimes \mathbb{C}$ 
we obtain the following equation; 
\begin{multline}\label{eq2}
\pi(\mathcal{J}(\xi_{M} - \sqrt{-1}(d\mu + \sqrt{-1}\alpha^{\xi}))) \\
= \frac{1}{2}\left( (J_{+} + J_{-})(\xi_{M}) - (\omega_{+}^{-1} - \omega_{+}^{-1})(-b(\xi_{M}) -\sqrt{-1}(d\mu^{\xi} + \sqrt{-1}\alpha^{\xi}))  \right). 
\end{multline}
By combining the equations (\ref{eq1}) and (\ref{eq2}), we see that 
the induced vector field 
$\xi_{M}$ can be written as 
\begin{equation}\label{eq3}
\xi_{M} = \frac{1}{2}
\Big( \omega_{+}^{-1}(d\mu^{\xi}) - \omega_{-}^{-1}(d\mu^{\xi}) \Big). 
\end{equation}
Let $\nabla$ be the Riemannian connection 
with respect to the metric $g$. 
Then by an easy calculation we have the following equality for 
$\xi_{M}^{\pm} := \omega_{\pm}^{-1}(d\mu^{\xi}) = -J_{\pm}g^{-1}(d\mu^{\xi})$; 
\begin{equation}\label{met}
g(\nabla^{2}\mu^{\xi}(Y), Z) = 
g((\nabla_{Y}J_{\pm})\xi_{M}^{\pm}, Z) + g(J_{\pm}(\nabla_{Y}\xi_{M}^{\pm}), Z). \end{equation}
Since the vector field $\xi_{M}^{\pm}$ vanishes on ${\rm Crit}(\mu^{\xi})$, 
the equation (\ref{met}) shows that 
\begin{equation}\label{nabla}
(\nabla^{2}\mu^{\xi})_{p}(Y_{p}) = J_{\pm}(\nabla_{Y_{p}}\xi_{M}^{\pm}) 
\end{equation}
for each $p \in {\rm Crit}(\mu^{\xi})$ and $Y_{p} \in T_{p}M$. 
Let $(L_{\xi})_{p}$ be an endomorphism on $T_{p}M$ 
defined by $(L_{\xi})_{p}(Y) := [\xi_{M}, Y]_{p} = -\nabla_{Y_{p}}\xi_{M}$. 
Then by the equation (\ref{eq3}) and (\ref{nabla}), we see that $(L_{\xi})_{p}$ can be written as 
\begin{equation}\label{LieHessian}
(L_{\xi})_{p} = -\frac{1}{2}(J_{+} - J_{-})(\nabla^{2}\mu^{\xi})_{p}. 
\end{equation}

Now we shall prove 
$T_{p}{\rm Crit}(\mu^{\xi}) = \ker (\nabla^{2}\mu^{\xi})_{p}$. 
Since each connected component ${\rm Crit}(\mu^{\xi})$ is 
a submanifold of $M$, it is clear that 
$T_{p}{\rm Crit}(\mu^{\xi}) \subset \ker (\nabla^{2}\mu^{\xi})_{p}$. 
Therefore we may only show that 
$\ker (\nabla^{2}\mu^{\xi})_{p} \subset T_{p}{\rm Crit}(\mu^{\xi})$. 
At first we have 
$\ker (\nabla^{2}\mu^{\xi})_{p} \subset \ker (L_{\xi})_{p}$ 
by the equation (\ref{LieHessian}). 
If we identify $(L_{\xi})_{p}$ with a vector field on $T_{p}M$, 
the one parameter family of diffeomorphisms 
$\{ (\exp t\xi _{*})_{p} \}_{t \in \mathbb{R}}$ on $T_{p}M$ coincides with 
$\{ \exp t(L_{\xi})_{p} \}_{t \in \mathbb{R}}$. Hence $\ker (L_{\xi})_{p}$ 
coincides with the fixed point set of 
$\{ (\exp t\xi _{*})_{p} \}_{t \in \mathbb{R}}$. 
Therefore we have 
\begin{equation*}
\ker (\nabla^{2}\mu^{\xi})_{p} \subset T_{p}{\rm Crit}(\mu^{\xi}), 
\end{equation*}
and this shows that 
$T_{p}{\rm Crit}(\mu^{\xi}) = \ker (\nabla^{2}\mu^{\xi})_{p}$. 
In particular, we see that $\mu^{\xi}$ is a Bott-Morse function. 

Finally, we shall show that the function $\mu^{\xi}$ has even 
index and coindex. By equation (\ref{nabla}), 
we see that 
\begin{equation*}
g((\nabla^{2}\mu^{\xi})_{p}(J_{\pm}Y), Z) 
= g((\nabla^{2}\mu^{\xi})_{p}Z, J_{\pm}Y) 
= g(J_{\pm}(\nabla_{Z}\xi_{M}^{\pm}), J_{\pm}Y) 
= g(\nabla_{Z}\xi_{M}^{\pm}, Y) 
\end{equation*}
for each $p \in {\rm Crit}(\mu^{\xi})$ and $Y, Z \in T_{p}M$. 
Since $\xi_{M} = \frac{1}{2}(\xi_{M}^{+} - \xi_{M}^{-})$ 
and $\xi_{M}$ is a Killing vector field, we obtain 
\begin{eqnarray*}
g((\nabla^{2}\mu^{\xi})_{p}(J_{+} - J_{-})(Y), Z) 
&=& g(\nabla_{Z}(\xi_{M}^{+} - \xi_{M}^{-}), Y) 
= 2g(\nabla_{Z}\xi_{M}, Y) \\
&=& -2g(\nabla_{Y}\xi_{M}, Z) 
= -g(\nabla_{Y}(\xi_{M}^{+} - \xi_{M}^{-}), Z) \\
&=& g((J_{+} - J_{-})(\nabla^{2}\mu^{\xi})_{p}(Y), Z). 
\end{eqnarray*}
Hence we see that $(\nabla^{2}\mu^{\xi})_{p}$ 
commutes with $J_{+} - J_{-}$ for all $p \in {\rm Crit}(\mu^{\xi})$. 
Now we define a differential $2$-form by $g(J_{+}-J_{-})$. 
Then since $g$ is positive definite and $J_{+}-J_{-}$ commutes with 
$(\nabla^{2}\mu^{\xi})_{p}$, $J_{+}-J_{-}$ preserves each eigenspace of $(\nabla^{2}\mu^{\xi})_{p}$ and hence $g(J_{+}-J_{-})$ is nondegenerate on each non-zero eigenspace of $(\nabla^{2}\mu^{\xi})_{p}$. 
Thus each non-zero eigenspace of $(\nabla^{2}\mu^{\xi})_{p}$ is 
even dimensional, in particular the index and coindex of 
the critical manifold are even. 
\end{proof}
\begin{rem}
The compactness assumption here is essential. 
If $M$ is noncompact, then a generalized moment map can not be 
seen a Bott-Morse function in general. Indeed, if we consider 
a trivial torus action on a complex manifold $(M, J)$, 
then the imaginary part of an arbitrary holomorphic 
function is a generalized moment map for this action. 
(See also Remark \ref{remlast}.)
\end{rem}
\subsection{A proof of Theorem A}
In this section we shall prove Theorem A. 
This proof involves induction over $m = \dim T^{m}$. 
Consider the statements: 
\begin{eqnarray*}
A_{m} &:& 
\text{``the level sets of $\mu$ are connected, for any $T^{m}$-action", }
{\rm and}\\
B_{m} &:& 
\text{``the image of $\mu$ is convex, for any $T^{m}$-action".}
\end{eqnarray*}
At first we see that $A_{1}$ holds by Proposition \ref{BM} 
and the fact that level sets of a Bott-Morse function on a connected 
compact manifold are connected 
if the critical manifolds have index and coindex 
$\not= 1$ (see \cite{MS} for example). The claim $B_{1}$ holds clearly 
because in $\mathbb{R}$ connectedness is equivalent to convexity. 

Now we prove $A_{m-1} \Rightarrow B_{m}$. Choose a matrix 
$A \in \mathbb{Z}^{m} \otimes \mathbb{Z}^{m-1}$ of maximal rank. 
If we identify $A$ with a linear mapping 
$A: \mathbb{R}^{m-1} \to \mathbb{R}^{m}$ and $T^{m}$ with 
$\mathbb{R}^{m}/\mathbb{Z}^{m}$, then 
$A$ induces an action of $T^{m-1}$ on $M$ by 
\begin{equation*}
\theta: p \mapsto (A \theta) \cdot p, 
\end{equation*}
for $\theta \in T^{m-1}$ and $p \in M$. 
The $T^{m-1}$-action 
is a Hamiltonian action with a generalized moment map 
$\mu_{A}(p) := A^{t}\mu(p)$ and a moment one form 
$\alpha_{A}^{\xi} := \alpha^{A\xi}$, where $A^{t}$ denotes the 
transpose of $A$. 

Given any $a \in \mu_{A}(M)$ and $p_{0} \in \mu_{A}^{-1}(a)$, 
we have the fiber of $\mu_{A}$ by 
\begin{equation*}
\mu_{A}^{-1}(a) = \{ p \in M\ |\ \mu(p) - \mu(p_{0}) \in \ker A^{t} \}. 
\end{equation*}
By the assumption $A_{m-1}$, $\mu_{A}^{-1}(a)$ is connected. 
Therefore, for each $p_{0}, p_{1} \in \mu_{A}^{-1}(a)$, if we connect $p_{0}$ to $p_{1}$ by a continuous path $p_{t}$ 
in $\mu_{A}^{-1}(a)$ 
we obtain a path $\mu(p_{t}) - \mu(p_{0})$ in $\ker A^{t}$. 
Since $A^{t}$ is surjective, 
$\ker A^{t}$ is $1$-dimensional. Hence $\mu(p_{t})$ must go through 
any convex combination of $\mu(p_{0})$ and $\mu(p_{1})$, 
which shows that any point on the line segment from 
$\mu(p_{0})$ to $\mu(p_{1})$ must be in $\mu(M)$. 

Any $p_{0}, p_{1} \in M$ with $\mu(p_{0}) \not= \mu(p_{1})$ can be 
approximated arbitrarily closely by points $p_{0}^{\prime}$ and 
$p_{1}^{\prime}$ with $\mu(p_{1}^{\prime}) - \mu(p_{0}^{\prime}) 
\in \ker A^{t}$ for some matrix 
$A \in \mathbb{Z}^{m} \otimes \mathbb{Z}^{m-1}$ 
of maximal rank. By the argument above, we see that the line 
segment from $\mu(p_{0}^{\prime})$ to $\mu(p_{1}^{\prime})$ 
must be in $\mu(M)$. 
By taking limits $p_{0}^{\prime} \to p_{0}$, and 
$p_{1}^{\prime} \to p_{1}$ we can conclude that $\mu(M)$ is
convex. 

Next we prove $A_{m-1} \Rightarrow A_{m}$. By identifying 
$\mathfrak{t}$ with $\mathbb{R}^{m}$, we can express the generalized 
moment map by $\mu = (\mu_{1}, \cdots, \mu_{m})$. 
We call the generalized moment map $\mu$ to be effective if 
the $1$-forms $d\mu_{1}, \cdots, d\mu_{m}$ are linearly independent.
Note that $p \in M$ is a regular point 
of $\mu$ if and only if $(d\mu_{1})_{p}, \cdots, (d\mu_{m})_{p}$ 
are linearly independent. If the generalized moment map $\mu$ is not 
effective, the action reduces to a Hamiltonian action of an 
$(m-1)$-dimensional subtorus. Indeed, If $\mu$ is not effective, 
there exists $c = (c_{1}, \cdots, c_{m}) \in \mathbb{R}^{m} \setminus \{ 0 \}$ 
such that $\sum_{i=1}^{m}c_{i}d\mu_{i} = 0$. Hence if we denote the 
canonical basis of $\mathfrak{t} \cong \mathbb{R}^{m}$ 
by $\xi_{1}, \cdots, \xi_{n}$, then we have 
\begin{equation*}
\sum_{i=1}^{m}c_{i} \Big( (\xi_{i})_{M} + \alpha_{i} \Big)
= \sum_{i=1}^{m}c_{i} 
\Big((\xi_{i})_{M} -\sqrt{-1}( d\mu_{i} + \sqrt{-1}\alpha_{i}) \Big)
\in L, 
\end{equation*}
where $\alpha = (\alpha_{1}, \cdots, \alpha_{m})$. 
Since $\sum_{i=1}^{m}c_{i} \left( (\xi_{i})_{M} + \alpha_{i} \right)$ 
is real and $L \cap \bar L = \{ 0 \}$, we obtain 
$\sum_{i=1}^{m}c_{i}(\xi_{i})_{M} = \sum_{i=1}^{m}c_{i}\alpha_{i} = 0$. 
Now consider a vector $\xi = \sum_{i=1}^{m}c_{i}\xi_{i} \in \mathfrak{t}$. 
By the same argument in the earlier part of the proof of 
Proposition \ref{BM}, we see that 
${\rm Crit}(\mu^{\xi}) = {\rm Fix}(T^{\xi})$ and hence 
the function $\mu^{\xi}$ is constant along $M$ because $\xi_{M} = 0$. 
For the simplicity, we may assume $\xi_{1}, \cdots, \xi_{m-1}, \xi$ are 
linearly independent. Then the $T^{m-1}$-action generated by 
$\xi_{1}, \cdots, \xi_{m-1}$ is a Hamiltonian action with 
a generalized moment map $\mu^{\prime} = (\mu_{1}, \cdots, \mu_{m-1})$ and 
a moment one form $\alpha^{\prime} = (\alpha_{1}, \cdots, \alpha_{m-1})$. 
Hence in this case the connectedness of fibers of $\mu$ follows from that 
of the reduced generalized moment map $\mu^{\prime}$. 
Hence we may assume that $\mu$ is effective. 
Then for each $\xi \in \mathfrak{t} \setminus \{ 0 \}$, $\mu^{\xi}$ is not 
a constant function, and the critical manifold ${\rm Crit}(\mu^{\xi})$ is 
an even dimensional proper submanifold. Now consider the union of 
critical manifolds 
\begin{equation*}
C = \cup_{\eta \in \mathfrak{t} \setminus \{ 0 \}}{\rm Crit}(\mu^{\eta}). 
\end{equation*}
We claim that the union $C$ is indeed a countable union of even dimensional 
proper submanifolds. To see this, recall that the critical points of 
$\mu^{\eta}$ are the fixed points of the action of the subtorus 
$T^{\eta} \subset T^{m}$ and form an even dimensional proper submanifold. 
Since the fixed point set decreases as the torus increases it suffices to 
consider $1$-dimensional subtorus or, equivalently, integer vectors $\eta$. 
This shows the assertion about $C$. 
In particular, $M \setminus C$ is a dense subset of $M$. 
In addition, since the condition $p \in M \setminus C$ 
is equivalent to the condition that 
$(d\mu_{1})_{1}, \cdots, (d\mu_{m})_{p}$ are linearly independent, 
we obtain $M \setminus C$ is open dense subset of $M$. 

\begin{lem}
The set of regular values of $\mu$ in $\mu(M)$ is a 
dense subset of $\mu(M)$. 
\end{lem}
\begin{proof}
For each $a = \mu(p) \in \mu(M)$, there exists a sequence 
$\{p_{i}\}_{i = 1}^{\infty} \subset M \setminus C$ 
which satisfies that $\lim_{i \to \infty}p_{i} = p$. Since $p_{i}$ is 
a regular point of $\mu$, $\mu(M)$ contains a neighborhood of 
$\mu(p_{i})$ by implicit function theorem. Moreover there exists 
a regular value $a_{i} \in \mathfrak{t}^{*}$ which is sufficiently 
close to $\mu(p_{i})$ and $\mu^{-1}(a_{i}) \not= \phi$ 
by Sard's theorem. Hence the sequence $\{ a_{i}\}_{i=1}^{\infty}$ 
approximates $a$. 
\end{proof}

By a similar argument, we see that the set of 
$a = (a_{1}, \cdots, a_{m}) \in \mathfrak{t}^{*}$ that 
$(a_{1}, \cdots, a_{m-1})$ is a regular value of 
$(\mu_{1}, \cdots, \mu_{m-1})$ in $\mu(M)$ is also 
a dense subset of $\mu(M)$. Hence, by continuity, to prove that 
$\mu^{-1}(a)$ is connected for every 
$a = (a_{1}, \cdots, a_{m}) \in \mathfrak{t}^{*}$, it suffics to 
prove that $\mu^{-1}(a)$ is connected whenever $(a_{1}, \cdots, a_{m-1})$ 
is a regular value for the reduced generalized moment map 
$(\mu_{1}, \cdots, \mu_{m-1})$. By the induction hypothesis, 
the submanifold 
\begin{equation*}
Q = \cap_{i=1}^{m-1}\mu_{i}^{-1}(a_{i}) 
\end{equation*}
is connected 
for a regular value $(a_{1}, \cdots, a_{m-1})$ of 
$(\mu_{1}, \cdots, \mu_{m-1})$. To complete the proof, we need the following 
lemma. 
\begin{lem}\label{BM2}
If $(a_{1}, \cdots, a_{m-1})$ is a regular value for 
$(\mu_{1}, \cdots, \mu_{m-1})$, the function 
$\mu_{m} : Q \to \mathbb{R}$ is a Bott-Morse function 
of even index and coindex. 
\end{lem}
\begin{proof}
By the hypothesis, $Q$ is a $2n-(m-1)$ dimensional connected submanifold 
of $M$. For each $p \in Q$, the subspace $W$ of the cotangent space 
$T_{p}^{*}M$ generated by $(d\mu_{1})_{p}, \cdots, (d\mu_{m-1})_{p}$ is 
$(m-1)$ dimensional because $p$ is regular. 
Therefore the tangent space $T_{p}Q$ of $Q$ coincides with the annihilator of $W$; 
\begin{equation*}
T_{p}Q = \{ X \in T_{p}M\ |\ \alpha(X) = 0\ (\forall \alpha \in W) \}. 
\end{equation*}
Hence we see that $p \in Q$ is a critical point of $\mu_{m}: Q \to \mathbb{R}$ 
if and only if there exists real numbers $c_{1}, \cdots, c_{m-1}$ 
such that 
\begin{equation*}
\sum_{i=1}^{m-1}c_{i}(d\mu_{i})_{p} + (d\mu_{m})_{p} = 0. 
\end{equation*}
This means that $p$ is a critical point of the function 
$\mu^{\xi} : M \to \mathbb{R}$, where 
$\xi = (c_{1}, \cdots, c_{m-1}, 1)$ $\in \mathfrak{t}$. 
By Proposition \ref{BM}, $\mu^{\xi}$ is a Bott-Morse function with 
even index and coindex. 
Furthermore, by Lemma \ref{Fix} and the fact ${\rm Crit}(\mu^{\xi}) = {\rm Fix}(T^{\xi})$, the critical set ${\rm Crit}(\mu^{\xi})$ is a finite union of generalized complex submanifolds. 
Now we shall prove the critical manifold 
${\rm Crit}(\mu^{\xi})$ intersects $Q$ transversally at $p$, that is, 
\begin{equation*}
T_{p}M = T_{p}{\rm Crit}(\mu^{\xi}) + T_{p}Q. 
\end{equation*}
For a subspace $S \subset T_{p}M$, we denote by $S^{0} \subset T_{p}^{*}M$ the annihilator of $S$; 
\begin{equation*}
S^{0} = \{ \alpha \in T_{p}^{*}M\ |\ \alpha(X) = 0\ (\forall X \in S) \}. 
\end{equation*}
Then since $(T_{p}Q)^{0} = W$, we obtain 
\begin{equation*}
(T_{p}{\rm Crit}(\mu^{\xi}) + T_{p}Q)^{0} 
= (T_{p}{\rm Crit}(\mu^{\xi}))^{0} \cap (T_{p}Q)^{0}
= (T_{p}{\rm Crit}(\mu^{\xi}))^{0} \cap W. 
\end{equation*}
Hence the critical manifold ${\rm Crit}(\mu^{\xi})$ 
intersects $Q$ transversally at $p$ if and only if 
$(T_{p}{\rm Crit}(\mu^{\xi}))^{0} \cap W = \{ 0 \}$. 
Thus we may only show that the differentials 
$(d\mu_{1})_{p}, \cdots$, $(d\mu_{m-1})_{p}$ remain linearly 
independent when restricted to the subspace $T_{p}{\rm Crit}(\mu^{\xi})$. 
Consider the vector fields 
$\xi_{1}^{+}, \cdots, \xi_{m-1}^{+}$ on $M$ defined by 
\begin{equation*}
d\mu_{i} = \omega_{+}(\xi_{i}^{+}),\quad i=1, \cdots, m-1. 
\end{equation*}
Since $\omega_{+} = gJ_{+}$, the vector field $\xi_{i}^{+}$ can be 
written as $\xi_{i}^{+} = -J_{+}g^{-1}(d\mu_{i})$. 
The $T^{m}$-invariance of the function $\mu_i$ implies 
\begin{equation*}
(\theta_{*})_{p}g^{-1}(d\mu_{i})_{p} 
= g^{-1}((\theta^{-1})^{*}d\mu_{i})_{p} 
= g^{-1}(d((\theta^{-1})^{*}\mu_{i}))_{p}
= g^{-1}(d\mu_{i})_{p}
\end{equation*}
for each $\theta \in T^{\xi}$. 
In particular, we see that the vector field $g^{-1}(d\mu_{i})$ is tangent to ${\rm Crit}(\mu^{\xi})$ because $T_{p}{\rm Crit}(\mu^{\xi}) = T_{p}{\rm Fix}(T^{\xi}) = \bigcap_{\theta \in T^{\xi}} \ker (1 - (\theta_{*})_{p})$. 
Moreover, since the critical manifold ${\rm Crit}(\mu^{\xi})$ 
is an almost complex submanifold of $(M, J_{+})$, the vector filed 
$\xi_{i}^{+} = -J_{+}g^{-1}(d\mu_{i})$ is also tangent to 
${\rm Crit}(\mu^{\xi})$. 
On the other hand, $(\xi_{1}^{+})_{p}, \cdots, (\xi_{m-1}^{+})_{p}$ 
are linearly independent on $T_{p}M$ because $p$ is regular. 
Hence they are also linearly independent on $T_{p}{\rm Crit}(\mu^{\xi})$. 
Since the $2$-form $\omega_{+}$ is still nondegenerate when 
it is restricted to ${\rm Crit}(\mu^{\xi})$, 
the $1$-forms $(d\mu_{1})_{p}, \cdots, (d\mu_{m-1})_{p}$ are linearly 
independent on $T_{p}^{*}{\rm Crit}(\mu^{\xi})$ and hence 
${\rm Crit}(\mu^{\xi})$ intersects $Q$ transversally as claimed. 
In particular, the critical set ${\rm Crit}(\mu_{m} |_{Q})$ of 
$\mu_{m} : Q \to \mathbb{R}$ is a finite union of submanifolds of $Q$ 
because ${\rm Crit}(\mu_{m} |_{Q}) = {\rm Crit}(\mu^{\xi}) \cap Q$. 

For each $X \in T_{p}M$ which is orthogonal to 
$T_{p}{\rm Crit}(\mu^{\xi})$, we have 
\begin{equation*}
(d\mu_{i})_{p}(X) 
= g_{p}(g^{-1}(d\mu_{i}), X)
= 0 
\end{equation*}
for $i = 1, \cdots, m-1$. 
This implies that the orthogonal complement 
$(T_{p}{\rm Crit}(\mu^{\xi}))^{\perp}$ of 
the subspace $T_{p}{\rm Crit}(\mu^{\xi})$ is contained in $T_{p}Q$. 
Hence the Hessian of $\mu^{\xi}$ at $p$ is nondegenerate on $T_{p}Q \cap (T_{p}{\rm Crit}(\mu^{\xi}))^{\perp} = (T_{p}{\rm Crit}(\mu^{\xi}))^{\perp}$
with even index and coindex. In other words, 
${\rm Crit}(\mu^{\xi}) \cap Q$ is the critical manifold of 
$\mu^{\xi}|_{Q}$ of even index and coindex. The same holds for 
$\mu_{m}|_{Q}$ since it only differs from $\mu^{\xi}$ by the 
constant $\sum_{i=1}^{m-1}c_{i}a_{i}$. Thus we have proved 
that the function $\mu_{m} : Q \to \mathbb{R}$ is 
a Bott-Morse function with even index and coindex. 
\end{proof}

By applying Lemma \ref{BM2}, if $(a_{1}, \cdots, a_{m-1})$ is a 
regular value for $(\mu_{1}, \cdots, \mu_{m-1})$, then the 
level set $\mu_{m}^{-1}(a_{m}) \cap Q = \mu^{-1}(a)$ is connected. 
This shows that $A_{m-1} \Rightarrow A_{m}$. 

Finally, we shall prove the third claim, that is, the image of the 
generalized moment map $\mu$ is the convex hull of the images 
of the fixed points of the action. 
By Lemma \ref{Fix}, the fixed point set ${\rm Fix}(T^{m})$ of the 
action decomposes into finitely many even dimensional connected 
submanifolds $C_{1}, \cdots, C_{N}$ of $M$. The generalized moment 
map $\mu$ is constant on each of these sets because 
$C_{i} \subset {\rm Crit}(\mu^{\xi})$ for $i = 1, \cdots, N$ and 
any $\xi \in \mathfrak{t}$. Hence there exists 
$a_{1}, \cdots, a_{N} \in \mathfrak{t}^{*}$ such that 
\begin{equation*}
\mu(C_{i}) = \{ a_{i} \},\quad i=1, \cdots, N. 
\end{equation*}
By what we have proved so far the convex hull of the 
points $a_{1}, \cdots, a_{N}$ is contained in $\mu(M)$. 
Conversely, let $a \in \mathfrak{t}^{*}$ be a point which 
is not in the convex hull of $a_{1}, \cdots, a_{N}$. Then 
there exists a vector $\xi \in \mathfrak{t}$ with rationally 
independent components such that 
\begin{equation*}
a_{i}(\xi) < a(\xi),\quad i=1, \cdots, N. 
\end{equation*}
Since the components of $\xi$ are rationally independent, 
we have ${\rm Crit}(\mu^{\xi}) = {\rm Fix}(T^{m})$. Hence 
the function $\mu^{\xi} : M \to \mathbb{R}$ 
attains its maximum on one of the sets $C_{1}, \cdots, C_{N}$. 
This implies 
\begin{equation*}
\sup_{p \in M}\mu^{\xi}(p) < a(\xi), 
\end{equation*}
and hence $a \not \in \mu(M)$. This shows that $\mu(M)$ 
is the convex hull of the points $a_{1}, \cdots, a_{N}$ and 
{\rm Theorem A} is proved. 

\begin{rem}
By applying the same arguments of 
our proof and Theorem 5.1 in \cite{Ler3}, Theorem A still holds 
in the case that $M$ is a compact connected $H$-twisted generalized complex 
orbifold. In this case, all connected components $C_{1}, \cdots, C_{N}$ of the critical set are connected generalized complex suborbifolds. 
\end{rem}
\section{Non-abelian convexity and connectedness properties}
The purpose of this section is to give a proof 
of Theorem B. Our proof is a simple generalization of 
the argument of Lerman, Meinrenken, Tolman and Woodward in 
\cite{Ler2} to generalized complex geometry. 
\subsection{Weak nondegeneracy of generalized moment maps}
In this subsection, we introduce an additional property ``weak nondegeneracy" 
for generalized moment maps, which always holds for compact cases. 
\begin{defn}\label{weak_nondegeneracy}
We say that a generalized moment map $\mu : M \to \mathfrak{g}^{*}$ has {\it weak nondegeneracy} if the following equality holds for all $\xi \in \mathfrak{g}$; 
\begin{equation*}
{\rm Crit}(\mu^{\xi}) = {\rm Fix}(T^{\xi}). 
\end{equation*}
\end{defn}

\begin{ex}
Let a compact Lie group $G$ act on a symplectic manifold $(M, \omega)$ 
in a Hamiltonian way with a moment map $\mu : M \to \mathfrak{g}^{*}$. 
Then the $G$-action on $(M, \mathcal{J}_{\omega})$ is Hamiltonian with 
a generalized moment map $\mu$ and  a  moment one form $\alpha = 0$. 
In this case, the generalized moment map $\mu$ has weak nondegeneracy. 
Indeed, since $d\mu^{\xi} = \iota_{\xi_{M}}\omega$ for each $\xi \in \mathfrak{g}$, it follows that $\xi_{M} = 0$ if and only if $d\mu^{\xi} = 0$. 
\end{ex}
\begin{ex}
Consider the trivial action of a compact torus $T^{m}$ on a complex manifold 
$(M, J)$. Then by identifying the Lie algebra $\mathfrak{t}$ with $\mathbb{R}^{m}$, each holomorphic map $h = (h_{1}, \cdots, h_{m}) : M \to \mathbb{C}^{m}$ defines a generalized moment map $\mu = {\rm Im}\ h$ and a moment one form $\alpha = d({\rm Re}\ h) = (d({\rm Re}\ h_{1}), \cdots, d({\rm Re}\ h_{m}))$ for the $T^{m}$-action, where ${\rm Re}\ h$ (resp. ${\rm Im}\ h$) denotes the real part (resp. the imaginary part) of $h$. In this case, $\mu$ has weak nondegeneracy if and only if $h$ is locally constant, because $\xi_{M}$ reduces to $0$ for all $\xi \in \mathfrak{t}$. 
\end{ex}

By the former part of the proof of Proposition \ref{BM}, we see that a generalized moment map for compact manifolds always has weak nondegeneracy. 
Moreover, the latter part of the proof of Proposition \ref{BM} 
tells us that, for noncompact manifolds, a generalized moment map having 
weak nondegeneracy is nondegenerate in the sense of abstract moment maps 
in Ginzburg-Guillemin-Karshon \cite{GGK}. 

\begin{rem}\label{rem41}
Let a compact Lie group $G$ act on an $H$-twisted generalized complex 
orbifold $(M, \mathcal{J})$ in a Hamiltonian way with a generalized 
moment map $\mu : M \to \mathfrak{g}^{*}$ and a moment one form 
$\alpha \in \Omega^{1}(M; \mathfrak{g}^{*})$. If $\mu$ has weak nondegeneracy, as in the case of symplectic orbifolds, the image of the differential of the generalized moment map at a point $p \in M$ is the annihilator of 
the corresponding isotropy Lie algebra $\mathfrak{g}_{p}$. 
In particular, we see that the following conditions are equivalent: 
\begin{enumerate}
\item $p \in M$ is a regular point of $\mu$. 
\item $\mathfrak{g}_{p} = \{ 0 \}$. 
\item The $G$-action at $p$ is locally free. 
\end{enumerate}
We prove the assertion here. 
Let $\tilde{T}_{p}M$ denote the uniformized tangent space of $M$ at $p$. 
For each $\xi \in \mathfrak{g}$, weak nondegeneracy condition of the generalized moment map implies that $(d\mu^{\xi})_{p} = 0$ if and only if $(\xi_{M})_{p} = 0$. 
Since 
\begin{equation}\label{diff}
(\mu_{*})_{p}(X)(\xi) = (d\mu^{\xi})_{p}(X) 
\end{equation}
for each $X \in \tilde{T}_{p}M$, we have $(\mu_{*})_{p}(X)(\xi) = 0$ 
for all $\xi \in \mathfrak{g}_{p}$. 
This shows that the image of $(\mu_{*})_{p}$ is contained in the annihilator $(\mathfrak{g}_{p})^{0}$. 
On the other hand, the equation (\ref{diff}) implies that 
$X \in \ker (\mu_{*})_{p}$ if and only if $(d\mu^{\xi})_{p}(X) = 0$ 
for all $\xi \in \mathfrak{g}$. Hence we obtain the equation $\ker (\mu_{*})_{p} = (D\mu)_{p}^{0}$, where $(D\mu)_{p}$ is the subspace of $\tilde{T}_{p}^{*}M$ generated by the differentials $(d\mu^{\xi})_{p}$ for all $\xi \in \mathfrak{g}$ and $(D\mu)_{p}^{0} \subset T_{p}M$ is its annihilator. 
In addition, since $\dim (D\mu)_{p} = \dim \mathfrak{g} - \dim \mathfrak{g}_{p}$ by weak nondegeneracy condition, 
we have $\dim \ker (\mu_{*})_{p} = \dim M - (\dim \mathfrak{g} - \dim \mathfrak{g}_{p})$. Hence we have 

\begin{equation*}
\dim (\mu_{*})_{p}(\tilde{T}_{p}M) 
= \dim \mathfrak{g} - \dim \mathfrak{g}_{p}
= \dim (\mathfrak{g}_{p})^{0}
\end{equation*}
and so $(\mu_{*})_{p}(\tilde{T}_{p}M) = (\mathfrak{g}_{p})^{0}$. 
This shows the assertion. 
In particular, the generalized moment map has constant rank on the principal stratum $M_{\rm prin}$, an open dense subset of $M$ defined to be the intersection of the set of the points of principal orbit type with the set of smooth points of $M$. (See \cite{GGK} for the definition of the principal orbit type.)
\end{rem}

\subsection{Generalized complex cuts}
In view of symplectic geometry, we introduce 
the notion of generalized complex cutting. 
Let $(M, \mathcal{J})$ be an $H$-twisted generalized complex orbifold 
which admits a Hamiltonian circle action with a generalized moment map 
$\mu : M \to \mathbb{R}$ and a moment one form $\alpha \in \Omega^{1}(M)$. 
We assume that the generalized moment map $\mu$ has weak nondegeneracy. 
For a regular value $\varepsilon \in \mathbb{R}$ of the generalized moment map, 
consider the disjoint union 
\begin{equation*}
M_{[\varepsilon, +\infty)} = 
\mu^{-1}((\varepsilon, +\infty)) \cup M_{\varepsilon} 
\end{equation*}
obtained from the orbifold with boundary $\mu^{-1}([\varepsilon, +\infty))$ 
by collapsing the boundary under the circle action. Then the disjoint union 
$M_{[\varepsilon, +\infty)}$ admits a natural structure of a twisted 
generalized complex orbifold. To see this, consider the product 
$M \times \mathbb{C}$ of the orbifold with a complex plane. 
It has a natural product $H$-twisted generalized complex structure: 
\begin{equation*}
\mathcal{J}_{M \times \mathbb{C}} = 
\left( \begin{array}{ccc}
  \mathcal{J} &    0 \\
  0           &    \mathcal{J}_{\omega} \\
\end{array}\right), 
\end{equation*}
where $\mathcal{J}_{\omega}$ is the natural generalized complex 
structure on $\mathbb{C}$ induced by the standard symplectic structure 
$\omega = (\sqrt{-1}/2)dz \wedge d\bar z$. The function 
$\nu : M \times \mathbb{C} \to \mathbb{R}$ given by 
$\nu(p, z) = \mu(p) - (1/2) |z|^{2}$ is a generalized moment map for 
the diagonal action of the circle, and the pull back of the moment one 
form $\alpha$ by the natural projection from $M \times \mathbb{C}$ 
to $M$ is a moment one form. 
Since $\mu$ has weak nondegeneracy, so does $\nu$. 
The point $\varepsilon \in \mathbb{R}$ is a regular value of $\nu$ if and only if it is a regular value of $\mu$. 
Moreover, the map 
\begin{equation*}
\{ p \in M\ |\ \mu(p) \geq \varepsilon \} \to \nu^{-1}(\varepsilon),\ 
p \mapsto (p, \sqrt{\mu(p) - \varepsilon}) 
\end{equation*}
induces a homeomorphism from $M_{[\varepsilon, +\infty)}$ 
to the reduced space $\nu^{-1}(\varepsilon)/S^{1}$. 
By weak nondegeneracy of the generalized moment map $\nu$, we see that 
the reduced space admits a natural $\tilde{H}$-twisted generalized complex structure. 
In particular, $M_{[\varepsilon, +\infty)}$ also admits a twisted generalized 
complex structure which is induced by the $\tilde{H}$-twisted 
generalized complex structure on the orbifold $\nu^{-1}(\varepsilon)/S^{1}$. 

\begin{defn}
We call the twisted generalized complex orbifold 
$M_{[\varepsilon, +\infty)}$ the generalized complex cut of $M$ 
with respect to the ray $[\varepsilon, +\infty)$. 
\end{defn}

The construction can be generalized to general torus actions as follows. 
Consider a Hamiltonian action of an $m$-dimensional torus $T^{m}$ on 
an $H$-twisted generalized complex orbifold $(M, \mathcal{J})$ with a 
generalized moment map $\mu : M \to \mathfrak{t}^{*}$ and a moment one form 
$\alpha \in \Omega^{1}(M; \mathfrak{t}^{*})$. 
We assume that the generalized moment map $\mu$ has weak nondegeneracy. 
Let $l \subset \mathfrak{t}$ denote the integral lattice. Choose $N$ vectors $v_{j} \in l,\ j=1, \cdots, N$. 
The endomorphism 
\begin{equation*}
\mathcal{J}_{M \times \mathbb{C}^{N}} = 
\left( \begin{array}{ccc}
  \mathcal{J} &    0 \\
  0           &    \mathcal{J}_{\omega} \\
\end{array}\right) 
\end{equation*}
is an $H$-twisted generalized complex structure on an orbifold 
$M \times \mathbb{C}^{N}$, where $\mathcal{J}_{\omega}$ is 
the natural generalized complex structure on $\mathbb{C}^{N}$ 
induced by the standard symplectic structure 
$\omega = (\sqrt{-1}/2)\sum_{i=1}^{N}dz^{i} \wedge d\bar z^{i}$. 
The map $\nu : M \times \mathbb{C}^{N} \to \mathbb{R}^{N}$ with 
$j$-th component 
\begin{equation*}
\nu_{j}(p, z) = \langle \mu(p), v_{j} \rangle -\frac{1}{2}|z_{j}|^{2} 
\end{equation*}
is a generalized moment map for the action of $T^{N}$ on 
$M \times \mathbb{C}^{N}$ induced by the Lie algebra homomorphism 
$\mathbb{R}^{N} \to \mathfrak{t},\ e_{j} \mapsto v_{j}$, where 
$\{ e_{1}, \cdots, e_{N} \}$ is the standard basis of $\mathbb{R}^{N}$. 
The $\mathbb{R}^{N}$-valued $1$-form $\beta$ with $j$-th component 
$\beta_{j}(p, z) = \langle \alpha(p), v_{j} \rangle$ is a moment one form. 
Because of weak nondegeneracy of $\mu$, the generalized moment map $\nu$ also has weak nondegeneracy. 
For each $b = (b_{1}, \cdots, b_{N}) \in \mathbb{R}^{N}$, we define a convex rational polyhedral set 
\begin{equation*}
P = 
\{ x \in \mathfrak{t}^{*}\ |\ 
\langle x, v_{j} \rangle \geq b_{j},\ j = 1, \cdots, N \}. 
\end{equation*}
The generalized complex cut of $M$ with respect to a rational polyhedral 
set $P$ is the reduction of $M \times \mathbb{C}^{N}$ at $b$. 
We denote it by $M_{P}$. If $b$ is a regular value of $\nu$, then 
$M_{P}$ is a twisted generalized complex orbifold by Remark \ref{rem41}. 
Note that regular values are generic by Sard's theorem. Furthermore 
if $P$ is a compact polytope, then the fact that $P$ is generic implies that $P$ is simple, that is, 
the number of codimension one faces meeting at a given vertex is the 
same as the dimension of $P$. 

A topological description of the cut space is given by the 
following result. This is a generalization of Proposition 2.4 
in \cite{Ler2} to generalized complex geometry and 
we can apply their proof of the theorem 
by replacing moment maps with generalized moment maps. 

\begin{prop}
Let an $m$-dimensional torus $T^{m}$ act on an $H$-twisted generalized complex 
orbifold $(M, \mathcal{J})$ effectively and in a Hamiltonian way 
with a generalized moment map $\mu : M \to \mathfrak{t}^{*}$ and a moment 
one form $\alpha \in \Omega^{1}(M; \mathfrak{t}^{*})$. Suppose that the generalized moment map $\mu$ has weak nondegeneracy. Consider a generic 
rational polyhedral set $P \subset \mathfrak{t}^{*}$ and the set of all 
open faces $\mathcal{F}_{P}$. Then the topological space $\tilde{M}_{P}$ defined by 
\begin{equation*}
\tilde{M}_{P} = \bigcup_{F \in \mathcal{F}_{P}}\mu^{-1}(F)/T_{F}, 
\end{equation*}
where $T_{F} \subset T^{m}$ is the subtorus of $T^{m}$ perpendicular to $F$, 
coincides with the generalized complex cut of $M$ with respect to $P$. 
In particular, $\tilde{M}_{P}$ is an $H$-twisted generalized complex orbifold 
with a natural Hamiltonian action of the torus $T^{m}$. Moreover, the map 
$\mu_{P} : \tilde{M}_{P} \to \mathfrak{t}^{*}$ induced by the 
restriction $\mu |_{\mu^{-1}(P)}$ is a generalized moment map, and the 
descending of the restriction $\alpha |_{\mu^{-1}(P)}$ of the moment one form 
is a moment one form for this action. Consequently, 
\begin{enumerate}
\item 
the cut space $\tilde{M}_{P}$ is connected if and only if $\mu^{-1}(P)$ is connected; 
\item 
the fibers of $\mu_{P}$ are connected if and only if 
fibers of $\mu |_{\mu^{-1}(P)}$ are connected; 
\item 
$\tilde{M}_{P}$ is compact if and only if $\mu^{-1}(P)$ is compact. 
\end{enumerate}
\end{prop}

Using the technique of generalized complex cuts, we can extend Theorem A 
to the case that $M$ is a noncompact orbifold and the generalized moment map has weak nondegeneracy. 
The proof is the same with the proof of Theorem 4.3 in \cite{Ler2}, 
except one must use the generalized complex cuts. 

\begin{thm}\label{noncompact}
Let an $m$-dimensional torus $T^{m}$ act on a connected $H$-twisted generalized complex orbifold $(M, \mathcal{J})$ in a Hamiltonian way with a generalized moment map $\mu : M \to \mathfrak{t}^{*}$ and a moment one form $\alpha \in \Omega^{1}(M; \mathfrak{t}^{*})$. 
If $\mu$ is proper as a map into a convex open set $U \subset \mathfrak{t}^{*}$ and has weak nondegeneracy, then 
\begin{enumerate}
\item
the image of $\mu$ is convex, 
\item
each fiber of $\mu$ is connected, and 
\item
if for every compact set $K \subset \mathfrak{t}^{*}$, 
the list of isotropy algebras for the $T^{m}$-action on $\mu^{-1}(K)$ 
is finite, then the image $\mu(M)$ is the intersection of $U$ 
with a rational locally polyhedral set. 
\end{enumerate}
\end{thm}

\subsection{The cross-section theorem}
Here we recall the notion of {\it slices} for group actions and 
prove a generalized complex geometry analogue of 
the cross-section theorem in symplectic geometry. 
\begin{defn}
Suppose that a group $G$ acts on an orbifold $M$. Given 
$p \in M$ with isotropy group $G_{p}$, a suborbifold 
$U \subset M$ containing $p$ is called a slice at $p$ if 
$U$ is $G_{p}$-invariant, $G \cdot U$ is a neighborhood of $p$, 
and the map 
\begin{equation*}
G \times_{G_{p}} U \to G \cdot U,\ [a, u] \mapsto a \cdot u 
\end{equation*}
is an isomorphism. 
\end{defn}
Consider the coadjoint action of a connected compact Lie group $G$ 
on $\mathfrak{g}^{*}$. For each $x \in \mathfrak{g}^{*}$, there is a 
unique largest open subset 
$U_{x} \subset \mathfrak{g}_{x}^{*} \subset \mathfrak{g}^{*}$ which is 
a slice at $x$. We call $U_{x}$ the natural slice at $x$ for the coadjoint 
action. A detailed construction can be seen in \cite{Ler2}. 
\begin{thm}[Cross-section]\label{cross2}
Let a compact connected Lie group $G$ act on an 
$H$-twisted generalized complex orbifold $(M, \mathcal{J})$ in a 
Hamiltonian way with a generalized moment map 
$\mu: M \to \mathfrak{g}^{*}$ and a moment one form 
$\alpha \in \Omega^{1}(M; \mathfrak{g}^{*})$. Consider the natural 
slice $U$ at $x \in \mathfrak{g}^{*}$ for the coadjoint action. Then 
the cross-section $R = \mu^{-1}(U)$ is a $G_{x}$-invariant 
generalized complex suborbifold of $M$, where $G_{x}$ is the isotropy 
group of $x$. Furthermore the $G_{x}$-action on $R$ is Hamiltonian with 
a generalized moment map $\mu_{R} := \mu |_{R}$ and a moment one form 
$\alpha |_{R}$, the restriction of $\alpha$ to $R$. 
\end{thm}
We shall give a proof of Theorem \ref{cross2} below. 
First note that since the slice $U$ is $G_{x}$-invariant and 
the generalized moment map $\mu$ is equivariant, the cross-section 
$R=\mu^{-1}(U)$ is also $G_{x}$-invariant. By definition of the 
slice, coadjoint orbits intersect $U$ transversally. 
Since the generalized moment map is equivariant, it is transversal 
to $U$ as well. Hence the cross-section is a suborbifold of $M$. 
We need to show that the cross-section $R$ is a generalized 
complex suborbifold of $M$. We shall show that 
$(L_{R})_{r} \subset (\tilde{T}_{r}R \oplus \tilde{T}^{*}_{r}R) 
\otimes \mathbb{C}$ defines an $i^{*}H$-twisted generalized complex 
structure of $R$, where $\tilde{T}_{r}R$ is the uniformized tangent space of 
$R$ at $r \in R$. Then we can see easily that $R$ is a generalized complex 
suborbifold. Consider a local representative $\varphi$ of $L$. 
If the pull back $(i^{*}\varphi)_{r}$ is a nondegenerate complex pure spinor, 
then it is a local representative of $L_{R}$ and 
hence $L_{R}$ defines an $i^{*}H$-twisted generalized complex structure. 
Hence we may only show that the pull back $(i^{*}\varphi)_{r}$ is a 
nondegenerate complex pure spinor of $R$ below. 

Since $\varphi_{r}$ is a nondegenerate complex pure spinor, there exists 
a decomposable complex $k$-form 
$\Omega \in \wedge^{k}\tilde{T}^{*}_{r}M \otimes \mathbb{C}$ and 
a complex $2$-form 
$B + \sqrt{-1}\omega \in \wedge^{2}\tilde{T}^{*}_{r}M \otimes \mathbb{C}$ 
such that 
\begin{equation*}
\varphi_{r} = \exp(B + \sqrt{-1}\omega) \wedge \Omega. 
\end{equation*}
The $2$-form $\omega$ is nondegenerate on the $2(n-k)$-dimensional subspace 
\begin{equation*}
S_{r} = \{ X \in \tilde{T}_{r}M\ |\ i_{X}(\Omega \wedge \bar \Omega) = 0 \}. 
\end{equation*}
Moreover, we claim that it satisfies that for each $\xi \in \mathfrak{g}$, 
$i_{(\xi_{M})_{r}}\omega = (d\mu^{\xi})_{r}$ on $S_{r}$. 
Indeed, since $\iota_{\xi_{M}}\varphi_{r} - 
\sqrt{-1}(d\mu^{\xi} + \sqrt{-1}\alpha^{\xi}) \wedge \varphi_{r} = 0$ 
by the definition of generalized moment maps, we have $\iota_{\xi_{M}}\Omega = 0$ and hence 
\begin{equation*}
\iota_{\xi_{M}}(B + \sqrt{-1}\omega) \wedge \Omega
= \sqrt{-1}(d\mu^{\xi} + \sqrt{-1}\alpha^{\xi}) \wedge \Omega. 
\end{equation*}
If we write $\Omega = \theta^{1} \wedge \cdots \wedge \theta^{k}$ by 
some $1$-forms $\theta^{1}, \cdots, \theta^{k} 
\in {\tilde T}^{*}_{r}M \otimes \mathbb{C}$, the vectors 
$\theta^{1}, \cdots, \theta^{k}, \bar{\theta^{1}}, \cdots, \bar{\theta^{k}}$ 
are linearly independent because the complex pure spinor 
$\varphi_{r} = \exp(B + \sqrt{-1}\omega) \wedge \Omega$ is nondegenerate. 
This implies that $\iota_{X}\Omega = 0$ and 
\begin{equation*}
\iota_{\xi_{M}}(B + \sqrt{-1}\omega)(X) \wedge \Omega
= \sqrt{-1}(d\mu^{\xi} + \sqrt{-1}\alpha^{\xi})(X) \wedge \Omega 
\end{equation*}
for each $X \in S_{r}$. Hence we obtain $i_{(\xi_{M})_{r}}\omega(X) = (d\mu^{\xi})_{r}(X)$
for each $\xi \in \mathfrak{g}$ and $X \in S_{r}$. 
This shows the claim. 

Consider the complex form on $\wedge^{\bullet}\tilde{T}^{*}_{r}R \otimes \mathbb{C}$ defined by 
\begin{equation*}
(i^{*}\varphi)_{r} = \exp(i^{*}B + \sqrt{-1}i^{*}\omega) \wedge i^{*}\Omega. 
\end{equation*}
To prove that $(i^{*}\varphi)_{r}$ is a nondegenerate complex pure spinor, 
we need to show the following statements: 
\begin{enumerate}
\item 
$i^{*}\Omega \wedge i^{*}\bar \Omega \not= 0$, 
in particular $i^{*}\Omega \not= 0$. 
\item 
$i^{*}\omega$ is nondegenerate on the subspace 
$S_{r}(R) = \{ X \in \tilde{T}_{r}R\ |\ i_{X}(i^{*}\Omega \wedge i^{*}\bar \Omega) = 0 \}$. 
\end{enumerate}

We first show the claim $1$. For the Lie algebra $\mathfrak{g}$ of $G$, 
$\mathfrak{g}_{x}$ denotes the Lie algebra of the stabilizer of 
$x \in \mathfrak{g}^{*}$. Then there exists a $G_{x}$-invariant subspace 
$\mathfrak{m}$ such that 
$\mathfrak{g} = \mathfrak{g}_{x} \oplus \mathfrak{m}$. 
For $y = \mu(r)$, the uniformized tangent space $\tilde{T}_{y}U$ is just the 
annihilator of $\mathfrak{m}$. 
Consider the subspace 
$\mathfrak{m}_{M}(r) = \{ (\xi_{M})_{r}\ |\ \xi \in \mathfrak{m} \}$. 
Note that $\mathfrak{m}_{M}(r) \subset S_{r}$ and 
$\dim \mathfrak{m}_{M}(r) = \dim \mathfrak{m}$. Now we show the 
following lemmata. 

\begin{lem}\label{lem41}
The subspace $\mathfrak{m}_{M}(r)$ is a symplectic vector space 
with respect to the $2$-form $\omega$ and is perpendicular to 
$S_{r} \cap \tilde{T}_{r}R$. 
\end{lem}
\begin{proof}
First observe that for $\xi \in \mathfrak{m}$ and 
$X \in S_{r} \cap \tilde{T}_{r}R$, 
\begin{equation*}
\omega((\xi_{M})_{r}, X) = ((\mu_{*})_{r}(X))(\xi) = 0
\end{equation*}
since $(\mu_{*})_{r}(X) \in T_{y}U = \mathfrak{m}^{\circ}$. 
Hence $\mathfrak{m}_{M}(r)$ is 
perpendicular to $S_{r} \cap \tilde{T}_{r}R$ with respect 
to the $2$-form $\omega$. 

Now we  show that the subspace $\mathfrak{m}_{M}(r)$ is 
a symplectic vector space. Since for $\xi, \eta \in \mathfrak{m}$, 
\begin{equation*}
\omega((\xi_{M})_{r}, (\eta_{M})_{r}) = ((\mu_{*})_{r}(\eta_{M}))(\xi) 
= ({\rm ad}^{*}(\eta)\mu(r))(\xi) = -y([\xi, \eta]), 
\end{equation*}
$\mathfrak{m}_{M}(r)$ is symplectic if and only if 
${\rm ad}^{*}(\mathfrak{m})y$ is a symplectic subspace of 
the tangent space $T_{y}(G \cdot y)$ of the coadjoint orbit $G \cdot y$. 
Since $G_{x} \cdot y \subset U$ and since $\mathfrak{m} = (T_{y}U)^{\circ}$, 
for each $\xi \in \mathfrak{m}$ and $\eta \in \mathfrak{g}_{x}$ we have 
\begin{equation*}
y([\xi, \eta]) = {\rm ad}^{*}(\eta)(y)\xi = 0, 
\end{equation*}
that is, $T_{y}(G_{x} \cdot y)$ and ${\rm ad}^{*}(\mathfrak{m})y$ are 
symplectically perpendicular in $T_{y}(G \cdot y)$. Hence it remains to show 
that the orbit $G_{x} \cdot y$ is a symplectic submanifold of the 
coadjoint orbit $G \cdot y$ because 
$T_{y}(G \cdot y) = T_{y}(G_{x} \cdot y) \oplus {\rm ad}^{*}(\mathfrak{m})y$. 
Since the natural projection 
${\rm pr} : \mathfrak{g}^{*} \to \mathfrak{g}_{x}^{*}$ 
is $G_{x}$-equivariant, we have ${\rm pr}(G_{x} \cdot y) = G_{x} \cdot {\rm pr}(y)$. 
By the definition of the symplectic forms on a coadjoint orbit 
the restriction of the symplectic form of $G \cdot y$ to $G_{x} \cdot y$ 
is just the pull-back by ${\rm pr}$ of the symplectic form 
of the $G_{x}$ coadjoint 
orbit $G_{x} \cdot {\rm pr}(y)$. 
Hence $G_{x} \cdot y$ is a symplectic submanifold of $G \cdot y$, and 
this proves the lemma. 
\end{proof}

\begin{lem}\label{lem42}
The uniformized tangent space $\tilde{T}_{r}M$ can be decomposed 
into the following direct sum: 
\begin{equation*}
\tilde{T}_{r}M = \tilde{T}_{r}R \oplus \mathfrak{m}_{M}(r). 
\end{equation*}
\end{lem}
\begin{proof}
If $X \in \tilde{T}_{r}R \cap \mathfrak{m}_{M}(r)$, then 
$X$ is perpendicular to $\mathfrak{m}_{M}(r)$ with respect to 
$\omega$ by Lemma \ref{lem41}. 
Since $\omega$ is nondegenerate on $\mathfrak{m}_{M}(r)$, 
we have $X = 0$ and hence 
$\tilde{T}_{r}R \cap \mathfrak{m}_{M}(r) = \{ 0 \}$. Furthermore, 
since $\dim R = \dim M - \dim \mathfrak{m}$, we see that 
$\dim \tilde{T}_{r}M = \dim \tilde{T}_{r}R + \dim \mathfrak{m}_{M}(r)$, 
and obtain the decomposition 
$\tilde{T}_{r}M = \tilde{T}_{r}R \oplus \mathfrak{m}_{M}(r)$. 
\end{proof}

The decomposition induces the decomposition of $S_{r}$; 
\begin{equation*}
S_{r} = (S_{r} \cap \tilde{T}_{r}R) \oplus \mathfrak{m}_{M}(r), 
\end{equation*}
because $\mathfrak{m}_{M}(r)$ is contained in $S_{r}$. Hence we have 
the dimension 
\begin{equation*}
\dim S_{r} \cap \tilde{T}_{r}R = \dim \tilde{T}_{r}R - 2k. 
\end{equation*}
This shows that we can choose a basis of $\tilde{T}_{r}M \otimes \mathbb{C}$; 
\begin{equation*}
e_{1}, \cdots, e_{a}, u_{1}, \cdots, u_{2k}, v_{1}, \cdots, v_{2(n-k)-a}, 
\end{equation*}
where 
$a = \dim S_{r} \cap \tilde{T}_{r}R$, 
$\{ e_{1}, \cdots, e_{a} \}$ is a basis of 
$S_{r} \cap \tilde{T}_{r}R$, 
$\{ e_{1}, \cdots, e_{a},\ u_{1}, \cdots, u_{2k} \}$ 
is a basis of $\tilde{T}_{r}R$ and
$\{ v_{1}, \cdots, v_{2(n-k)-a} \}$ 
is a basis of $\mathfrak{m}_{M}(r)$. 
Since $e_{i}, v_{j} \in S_{r}$, we have 
$i_{e_{i}}(\Omega \wedge \bar \Omega) = 
i_{v_{j}}(\Omega \wedge \bar \Omega) =0$. 
Hence we see that 
$\Omega \wedge \bar \Omega (u_{1}, \cdots, u_{2k}) \not= 0$ 
because $\Omega \wedge \bar \Omega \not=0$ on $\tilde{T}_{r}M$. 
This shows that $i^{*}\Omega \wedge i^{*} \bar \Omega \not= 0$, 
and hence we have proved the claim $1$. 

Now we prove the claim $2$. We can check easily that 
$S_{r} \cap \tilde{T}_{r}R \subset S_{r}(R)$. 
Since $i^{*}\Omega \wedge i^{*} \bar \Omega \not= 0$, we have 
\begin{equation*}
\dim (S_{r} \cap \tilde{T}_{r}R) = \dim S_{r}(R) = \dim R -2k. 
\end{equation*}
Hence we obtain the equation $S_{r} \cap \tilde{T}_{r}R = S_{r}(R)$. 
Now take a vector $X \in S_{r}(R)$ which 
is perpendicular to $S_{r}(R)$ with respect to $\omega$, that is, 
$\omega(X, Y) = 0$ for any $Y \in S_{r}(R)$. Then since 
$\omega(X, (\xi_{M})_{r}) = 0$ for any $\xi \in \mathfrak{m}$, 
we see that $\omega(X, Y) = 0$ for any $Y \in S_{r}$. 
Since $\omega$ is nondegenerate on $S_{r}$, we have $X = 0$ 
and hence $\omega$ is also nondegenerate on $S_{r}(R)$. This proves the 
claim $2$. 

By the claims $1$ and $2$, we see that $(i^{*}\varphi)_{r}$ is a nondegenerate 
complex pure spinor and that $R$ is a generalized complex suborbifold of $M$. 
Finally, it is clear that the $G_{x}$-action on $R$ preserves the induced 
$i^{*}H$-twisted generalized complex structure and is Hamiltonian 
with a generalized moment map $\mu_{R} = \mu |_{R}$ and the moment one form 
$i^{*}\alpha$. This completes the proof of Theorem \ref{cross2}. 
\subsection{A proof of Theorem B}
By Remark \ref{rem41} and Theorem \ref{cross2}, we can extend Theorem 3.1 in Lerman-Meinrenken-Tolman-Woodward \cite{Ler2} to generalized complex geometry, 
and we see that there is a unique open face $\sigma$ of the Weyl chamber $\mathfrak{t}_{+}^{*}$ such that 
\begin{enumerate}
\item 
$\mu(M) \cap \sigma$ is dense in $\mu(M) \cap \mathfrak{t}_{+}^{*}$, 
\item 
the preimage $Y = \mu^{-1}(\sigma)$ is a connected $T$-invariant 
generalized complex suborbifold of $M$, and the restriction 
$\mu_{Y} = \mu |_{Y}$ and the pull back of $\alpha$ to $Y$ are 
a generalized moment map and a moment one form for the action of 
the maximal torus $T$, and
\item 
the set $G \cdot Y$ is dense in $M$. 
\end{enumerate}
(See the proof of Theorem 3.1 in \cite{Ler2}.) Since $\mu$ is proper, 
the restriction $\mu_{Y} : Y \to \mathfrak{t}^{*}$ is proper as a map 
into the open convex set $\sigma$. By Theorem \ref{noncompact}, 
the image $\mu(Y)$ is convex and is the intersection of $\sigma$ with a 
locally polyhedral set $P$, that is, $\mu(Y) = \sigma \cap P$. 
Therefore we have $\mu(M) \cap \mathfrak{t}_{+}^{*} = \overline{\mu(Y)}$. 
Since the closure of a convex set is also convex, the moment set 
$\mu(M) \cap \mathfrak{t}_{+}^{*}$ is convex. Moreover, since 
$\mu(M) \cap \mathfrak{t}_{+}^{*} = \overline{\sigma \cap P} 
= \bar \sigma \cap P$, $\mu(M) \cap \mathfrak{t}_{+}^{*}$ is a 
locally polyhedral set. Thus we have proved the first assertion. 
Now we shall show that the fiber $\mu^{-1}(x)$ is connected for all 
$x \in \mathfrak{g}^{*}$. We may assume $x \in \mathfrak{t}^{*}_{+}$. 
Since the fiber of $\mu |_{Y}$ is connected, the fiber of the 
restriction $\mu |_{G \cdot Y}$ is also connected. 
Observe that since $\mu^{-1}(G \cdot x)/G = \mu^{-1}(x)/G_{x}$ and 
the groups $G$ and $G_{x}$ are connected, the connectedness of $\mu^{-1}(x)$ is equivalent to that of $\mu^{-1}(G \cdot x)$. 
To prove the connectedness of 
$\mu^{-1}(G \cdot x)$, it is suffices to show that for any 
convex open neighborhood $B$ of $x$ in $\mathfrak{t}_{+}^{*}$, 
the closure of the open set $\mu^{-1}(G \cdot (B \cap \mathfrak{t}_{+}^{*}))$ 
is connected. By the condition $3$ of the open face $\sigma$, the 
intersection 
$\mu^{-1}(G \cdot (B \cap \mathfrak{t}_{+}^{*})) \cap G \cdot Y 
= G \cdot \mu^{-1}(B \cap \sigma)$ is dense in 
$\mu^{-1}(G \cdot (B \cap \mathfrak{t}_{+}^{*}))$ and hence also dense 
in its closure. Since $B \cap \sigma \cap \mu(M)$ is convex 
and $\mu^{-1}(y)$ is connected for each $y \in \sigma$, the set 
$G \cdot \mu^{-1}(B \cap \sigma)$ is connected and therefore its closure 
is also connected. This completes the proof of Theorem B. 

\begin{rem}\label{remlast}
In noncompact cases, the assumption of weak nondegeneracy is essential for the convexity property. For instance, consider the trivial action of $3$-dimensional 
compact torus $G = T^{3}$ on a complex manifold $M = \mathbb{C}$ 
with the standard complex structure $J$ and a holomorphic map 
$h : M \to \mathbb{C}^{3}$ defined by 
\begin{equation*}
h(z) = (\sqrt{-1}z, z, z^{2}). 
\end{equation*}
Then by identifying the Lie algebra $\mathfrak{t}$ with $\mathbb{R}^{3}$, 
we see that the action is a Hamiltonian action on a generalized complex manifold $(M, \mathcal{J}_{J})$ with a generalized moment map 
\begin{equation*}
\mu = {\rm Im}\ h = (x, y, 2xy) 
\end{equation*}
and a moment one form 
\begin{equation*}
\alpha = d({\rm Re}\ h) = (-dy, dx, 2xdx - 2ydy), 
\end{equation*}
where $z = x + \sqrt{-1}y$. 
Since the natural identification ${\rm id}: M \to \mathbb{R}^{2}$ is proper, 
the generalized moment map $\mu : M \to \mathbb{R}^{3}$ is also proper. 
In addition, $\mu$ does not have weak nondegeneracy 
because for $\xi = (1, 0, 0) \in \mathfrak{t}$, we have 
$d\mu^{\xi} = dx$ and hence ${\rm Crit}(\mu^{\xi}) = \phi$. 
(Note that ${\rm Fix}(T^{\xi}) = M$ for all $\xi \in \mathfrak{t}$ since the $G$-action is trivial.) 
In this case, the convexity property of the generalized moment map 
does not hold. Indeed, the image of the generalized moment map $\mu$ is 
just the graph of the function of two variables $f(x, y) = 2xy$. 
\end{rem}

\subsection{Concluding remarks}
A concept of generalized complex structures arises naturally 
when we consider a deformation of symplectic structures. 
Then we can consider a Hamiltonian action on a generalized complex 
manifold as a family of Hamiltonian actions of symplectic manifolds. 
We shall give a simple example below. 

Let $\mathbb{CP}^{2}$ be a $2$-dimensional complex projective space 
with the homogeneous coordinates $[z_{0}: z_{1}: z_{2}]$, and 
$\omega_{{\rm F. S. }}$ the Fubini-Study metric on $\mathbb{CP}^{2}$. 
For each $w = (w_{1}, w_{2}) \in \mathbb{C}^{*} \times \mathbb{C}^{*}$ 
we define a projective transformation $T_{w} \in {\rm PGL}(3, \mathbb{C})$ by 
\begin{equation*}
T_{w}([z_{0}: z_{1}: z_{2}]) = [z_{0}: |w_{1}|z_{1}: |w_{2}|z_{2}]. 
\end{equation*}
Then we have a deformation of the Fubini-Study metric 
$T_{w}^{*}\omega_{{\rm F. S. }}$. 
Consider the $T^{2}$-action on $\mathbb{CP}^{2}$ defined by 
\begin{equation*}
(\theta_{1}, \theta_{2}) \cdot [z_{0}: z_{1}: z_{2}] 
= [z_{0}: \theta_{1}z_{1}: \theta_{2}z_{2}] 
\end{equation*}
for all $(\theta_{1}, \theta_{2}) \in T^{2}$. Since the transformation 
$T_{w}$ commutes with the $T^{2}$-action, the action on a symplectic manifold 
$(\mathbb{CP}^{2}, T_{w}^{*}\omega_{{\rm F. S. }})$ is Hamiltonian 
with a moment map 
\begin{equation*}
\mu_{w}([z_{0}: z_{1}: z_{2}]) 
= -\frac{1}{2|z|^{2}}(|w_{1}| \cdot |z_{1}|^{2}, |w_{2}| \cdot |z_{2}|^{2}). 
\end{equation*}
By symplectic convexity theorem, we see that the image $\Delta_{w}$ of 
the moment map $\mu_{w}$ is the convex hull of 
$\{ (0, 0), (-|w_{1}|/2, 0), (0, -|w_{2}|/2) \}$, which is of course a 
compact polytope. 

Here we have obtained a family of Hamiltonian actions on symplectic 
manifolds. By considering a generalized complex structure, 
we can treat them at once. Consider the product 
$M = (\mathbb{C}^{*})^{2} \times \mathbb{CP}^{2}$ of 
an algebraic torus with a projective space. 
Since the $2$-form $T_{w}^{*}\omega_{{\rm F. S. }}$ 
is a symplectic form of $\mathbb{CP}^{2}$ for each 
$w \in (\mathbb{C}^{*})^{2}$, we can define a complex pure spinor 
$\varphi$ on $M$ by 
\begin{equation*}
\varphi = 
dw_{1} \wedge dw_{2} \wedge \exp \sqrt{-1}T_{w}^{*}\omega_{{\rm F. S. }}. 
\end{equation*}
Furthermore, since the complex pure spinor $\varphi$ is nondegenerate, it defines a generalized complex structure $\mathcal{J}_{\varphi}$ on $M$. 

Now consider a $T^{2}$-action on a generalized complex manifold 
$(M, \mathcal{J}_{\varphi})$ defined by lifting the $T^{2}$-action on 
$\mathbb{CP}^{2}$ to $M$; 
\begin{equation*}
(\theta_{1}, \theta_{2}) \cdot (w, [z_{0}: z_{1}: z_{2}]) 
= (w, [z_{0}: \theta_{1}z_{1}: \theta_{2}z_{2}]), 
\end{equation*}
for each $(\theta_{1}, \theta_{2}) \in T^{2}$. 
The $T^{2}$-action on $(M, \mathcal{J}_{\varphi})$ is Hamiltonian 
with a generalized moment map 
\begin{equation*}
\mu (w, [z_{0}: z_{1}: z_{2}]) 
= \mu_{w}([z_{0}: z_{1}: z_{2}]) 
\end{equation*}
and a moment one form $\alpha = 0$. The image $\Delta$ of 
the generalized moment map $\mu$ is a convex polyhedral set, 
\begin{equation*}
\Delta = \{ (x, y) \in \mathbb{R}^{2}\ |\ x \leq 0,\ y \leq 0 \}. 
\end{equation*}
When we restrict the $T^{2}$-action to the fiber 
$M_{w} = \{ w \} \times \mathbb{CP}^{2} \cong \mathbb{CP}^{2}$, the action on $M_{w}$ is equivalent to the Hamiltonian $T^{2}$-action on $\mathbb{CP}^{2}$ and 
the generalized moment map $\mu$ restricted to $M_{w}$ coincides with 
the moment map $\mu_{w}$. 
This shows that we can think of the Hamiltonian $T^{2}$-action 
on a generalized complex manifold $(M, \mathcal{J}_{\varphi})$ as a family 
of Hamiltonian $T^{2}$-actions on symplectic manifolds 
$(\mathbb{CP}^{2}, T_{w}^{*}\omega_{{\rm F. S. }})$. Then the image $\Delta$ of the generalized moment map $\mu$ coincides with the union of $\Delta_{w}$; 
\begin{equation*}
\Delta = \cup_{w \in (\mathbb{C}^{*})^{2}}\Delta_{w}. 
\end{equation*}
Here we see that not only each $\Delta_{w}$ is convex, but 
the union $\Delta$ is also convex. Note that $\Delta$ is not compact although $\Delta_{w}$ is a compact polytope for each $w \in (\mathbb{C}^{*})^{2}$. 

\end{document}